\documentclass{amsart}%
\usepackage{amssymb}
\usepackage{amsmath}
\usepackage{amsfonts}
\usepackage[all,cmtip]{xy}
\usepackage[margin=1.5in]{geometry}
\usepackage{hyperref}
\usepackage{graphicx}%
\newtheorem{theorem}{Theorem}
\newtheorem{theoremannounce}{Theorem}
\numberwithin{theorem}{section}
\theoremstyle{plain}

\newtheorem{definition}[theorem]{Definition}

\newtheorem{lemma}[theorem]{Lemma}
\newtheorem{lemmaconst}[theorem]{Lemma/Construction}

\newtheorem{problem}{Problem}

\newtheorem{proposition}[theorem]{Proposition}
\theoremstyle{remark}

\newtheorem{example}[theorem]{Example}

\numberwithin{equation}{section}

\begin{document}
\title[Quinn's formula and abelian $3$-cocycles]{Quinn's formula and abelian\linebreak$3$-cocycles for quadratic forms}
\author{Oliver Braunling}
\address{Mathematical Institute, University of Freiburg, Ernst-Zermelo-Strasse 1,
79104, Freiburg im Breisgau, Germany}
\email{oliver.braeunling@math.uni-freiburg.de}
\thanks{The author was supported by DFG GK1821 \textquotedblleft Cohomological Methods
in Geometry\textquotedblright.}

\begin{abstract}
In pointed braided fusion categories knowing the self-symmetry braiding of
simples is theoretically enough to reconstruct the associator and braiding on
the entire category (up to twisting by a braided monoidal auto-equivalence).
We address the problem to provide explicit associator formulas given only such
input. This problem was solved by Quinn in the case of finitely many simples.
We reprove and generalize this in various ways. In particular, we show that
extra symmetries of Quinn's associator can still be arranged to hold in
situations where one has infinitely many isoclasses of simples.

\end{abstract}
\maketitle

\section{Introduction}

This note applies to both (a) pointed braided fusion categories as well as (b)
braided categorical groups. Both are special types of braided monoidal
categories. Both settings are closely related, yet a little different. We hope
that we have found a way to formulate the introduction so that it is clear,
irrespective of which of these applications the reader might have in mind.
Among our four results below, Theorem \ref{thmann_GeneralizedQuinn} and
Theorem \ref{thmann_ExponentialFormatThreeCocycle} mainly reprove known
results differently, while Theorem \ref{thmann_MainCocycleProducer} and
Theorem \ref{thm_annA} appear to be new.\medskip

The problem motivating this note is the following: Suppose $(\mathsf{C}%
,\otimes)$ is a braided monoidal category of type either as in (a) or (b).
Then this category has an associator%
\[
a_{X,Y,Z}:(X\otimes Y)\otimes Z\overset{\sim}{\longrightarrow}X\otimes
(Y\otimes Z)
\]
and a braiding%
\[
s_{X,Y}:X\otimes Y\overset{\sim}{\longrightarrow}Y\otimes X
\]
as part of its braided monoidal structure.\ One may find a different
associator/braiding for the same bifunctor $\otimes:\mathsf{C}\times
\mathsf{C}\rightarrow\mathsf{C}$ such that the identity functor
$\operatorname*{id}:\mathsf{C}\rightarrow\mathsf{C}$ can be promoted to a
braided monoidal functor\footnote{in the setting (a) we tacitly assume these
to be $k$-linear, where $k$ is the base field of the fusion category.}. The
\textquotedblleft orbit\textquotedblright\ of all associators/braidings which
can obtained by twisting with such braided monoidal self-equivalences is
pinned down by very little data. In concrete terms: In the setting of (a)
suppose $X$ is a simple object, resp. in the setting of (b) an arbitrary
object. It has the self-symmetry\ braiding $s_{X,X}:X\otimes X\overset{\sim
}{\longrightarrow}X\otimes X$, and moreover $X$ is invertible in $\mathsf{C}$,
so we have the functor of tensoring with its inverse, $(-)\mapsto(-)\otimes
X^{-1}$. Thus, by functoriality, we get a well-defined automorphism%
\begin{equation}
s_{X,X}\otimes X^{-1}\otimes X^{-1}:1_{\mathsf{C}}\overset{\sim}%
{\longrightarrow}1_{\mathsf{C}} \label{lbb1}%
\end{equation}
of the monoidal unit $1_{\mathsf{C}}$. Hence, this is an element of the
abelian group $\pi_{1}(\mathsf{C},\otimes)=\operatorname{Aut}_{\mathsf{C}%
}(1_{\mathsf{C}})$. It turns out that this element only depends on the class
of the object $X$ in $\pi_{0}(\mathsf{C}_{\operatorname*{simp}},\otimes)$,
where in the setting of (a) $\mathsf{C}_{\operatorname*{simp}}$ is the
groupoid of simple objects in $\mathsf{C}$, and for (b) let $\mathsf{C}%
_{\operatorname*{simp}}:=\mathsf{C}$ be the entire category. Thus, we get a
map%
\[
q:\pi_{0}(\mathsf{C}_{\operatorname*{simp}},\otimes)\longrightarrow\pi
_{1}(\mathsf{C}_{\operatorname*{simp}},\otimes)\text{.}%
\]
One can show that this map is a quadratic form. On the one hand, this form
does not change under the aforementioned braided monoidal self-equivalences of
$\mathsf{C}$ (which induce the identity map on $\pi_{0}$ and $\pi_{1}$). On
the other hand, it also distinguishes different such orbits, i.e. it is a
complete invariant.

This note is concerned with the problem to provide an explicit formula for the
associator and braiding for a given quadratic form $q$, i.e. if we only know
the self-symmetry\ braidings of Equation \ref{lbb1}, but have possibly no clue
and no candidates what the associator and braiding should do on general
objects $X,Y,Z$. This is a kind of integration problem: Find a valid choice of
associators and braidings for the entire category such that, restricted to
self-symmetries, it agrees with the given quadratic form.

Even stronger: One may ask whether there is a \textquotedblleft simplest
choice\textquotedblright\ of associators, e.g., an associator with additional
symmetries (we propose a possible solution in Theorem \ref{thm_annA}).\medskip

The problem can be attacked in concrete form as follows: First, since we only
work up to braided monoidal equivalences inducing the identity map on $\pi
_{0}$ and $\pi_{1}$, it suffices to work with a skeleton of the category. Here
the datum of an associator and braiding is encoded in an abelian $3$-cocycle%
\[
\lbrack(h,c)]\in H_{ab}^{3}(G,M)
\]
for $G:=\pi_{0}(\mathsf{C}_{\operatorname*{simp}},\otimes)$, $M:=\pi
_{1}(\mathsf{C}_{\operatorname*{simp}},\otimes)$. In concrete terms, a cocycle
representative has the form of maps%
\[
h:G\times G\times G\longrightarrow M\qquad\text{and}\qquad c:G\times
G\longrightarrow M
\]
corresponding to the associator and braiding. Note that in general $h$ or $c$
cannot be taken multilinear. It is more complicated than that.

For general abelian groups $G,M$, Eilenberg and Mac Lane have constructed an
isomorphism%
\begin{equation}
\operatorname*{tr}:H_{ab}^{3}(G,M)\overset{\sim}{\longrightarrow
}\operatorname*{Quad}(G,M)\text{,} \label{lbjca1}%
\end{equation}
showing that this cohomology group is isomorphic to the group of quadratic
forms on $G$ with values in $M$, \cite{MR0045115}, \cite{MR0056295}. This
isomorphism underlies the above reconstruction procedure. Given only the
\textquotedblleft self-symmetries\textquotedblright, i.e. only the right side,
finding an associator and braiding amounts to finding a preimage under this
isomorphism. (In particular, this paper provides a new proof of the
surjectivity of the map in Equation \ref{lbjca1} as a side result, see
\S \ref{sect_ProofOfEilenbergMacLaneThm})

\subsection{Application to explicit associator formulas}

Unfortunately, solving the integration problem is not entirely trivial. The
classical proof for the isomorphism in Equation \ref{lbjca1} goes as follows:
Do the cases $G=\mathbb{Z}$ and $G=\mathbb{Z}/n\mathbb{Z}$ individually, then
use that both sides of Equation \ref{lbjca1} are quadratic functors in $G$,
and exploit that any abelian group is a colimit of finitely generated ones.
This sounds deceivingly simple, but note that if one wants an explicit
formula, one needs a cocycle formula, i.e. a lift%
\begin{equation}%
\xymatrix{
Z_{ab}^{3}(G,M) \ar@{->>}[d] \\
H_{ab}^{3}(G,M) & \operatorname{Quad}(G,M) \ar_{\cong}^{\operatorname{tr}^{-1}
}[l] \ar@{-->}[ul]_{?}
}
\label{l_fig1}%
\end{equation}
and once you consider such lifts, the nice functoriality properties break (as
an illustration: A\ compatible family of cohomology classes in some directed
system underlying the colimit need not come from a compatible family of
cocycles, because there is no a priori control of the coboundaries along the
system). We solve this as follows:

In \S \ref{sect_AdmissiblePresentation}-\ref{sect_Lift} we shall introduce the
concept of `\textit{optimal admissible presentations}' and `\textit{admissible
liftings}', and this can be thought of as machines to produce explicit cocycle
formulas. Since these are in general necessarily non-linear maps and highly
non-unique, it should not be surprising that admissible presentations and
liftings also allow for a lot of variation.

\begin{theoremannounce}
\label{thmann_MainCocycleProducer}Let $G,M$ be abelian groups and
$q\in\operatorname*{Quad}(G,M)$. Suppose

\begin{itemize}
\item $(F_{0},\pi,C)$ is an optimal admissible presentation (see
\S \ref{sect_AdmissiblePresentation}), and

\item $\widetilde{(-)}$ is an admissible lifting (see \S \ref{sect_Lift}).
\end{itemize}

Then%
\[
h(x,y,z):=-C(\widetilde{x},L(y,z))\text{,}\qquad c(x,y):=C(\widetilde
{x},\widetilde{y})\text{,}%
\]
with the non-linear function $L(x,y):=\widetilde{(x+y)}-\widetilde
{x}-\widetilde{y}$, defines an abelian $3$-cocycle whose attached quadratic
form is $q$. This gives a concrete lift as in Figure \ref{l_fig1} for the
inverse Eilenberg--Mac Lane isomorphism.
\end{theoremannounce}

This will be Theorem \ref{thm_GeneralFormulaWithAP}. It underlies all other
results of the paper. Quinn \cite{MR1734419} has given an explicit formula in
the case $G$ finite abelian and $M:=R^{\times}$ for $R$ a commutative ring. We
give a new proof for his formula as an application of Theorem
\ref{thmann_MainCocycleProducer}. In our generalized version $M$ is arbitrary
and $G$ is allowed to be of a more general form. In particular, all finitely
generated abelian groups are covered by our version:

\begin{theoremannounce}
[Generalized Quinn formula]\label{thmann_GeneralizedQuinn}Let $M$ be any
abelian group. Suppose%
\[
G=\left(  \bigoplus_{j\in J_{1}}\mathbb{Z}\right)  \oplus\left(
\bigoplus_{j\in J_{2}}\mathbb{Z}/n_{j}\mathbb{Z}\right)
\]
for $J_{1},J_{2}$ any index sets, and $n_{j}\geq1$ integers. Fix a total order
on the disjoint union $J:=J_{1}\dot{\cup}J_{2}$, say with $J_{1}<J_{2}$. Write
$(\mathsf{e}_{j})_{j\in J}$ for the generator $1$ in the $j$-th cyclic group.
Let $q\in\operatorname*{Quad}(G,M)$ be a quadratic form and
$b(x,y):=q(x+y)-q(x)-q(y)$ its polarization. Define%
\[
\sigma_{i,j}:=\left\{
\begin{array}
[c]{ll}%
b(\mathsf{e}_{i},\mathsf{e}_{j}) & \text{if }i<j\\
q(\mathsf{e}_{i}) & \text{if }i=j\\
0 & \text{if }i>j\text{.}%
\end{array}
\right.
\]
Then the pair $(h,c)$ with%
\[
h(x,y,z):=\sum_{\substack{j\in J_{2}\\\operatorname*{with}\text{ }y_{j}%
+z_{j}\geq n_{j}}}x_{j}n_{j}\sigma_{j,j}\qquad\text{and}\qquad c(x,y):=\sum
_{\substack{i,j\in J\\\operatorname*{with}\text{ }i\leq j}}x_{i}y_{j}%
\sigma_{i,j}%
\]
defines an abelian $3$-cocycle such that the trace map of Equation
\ref{lbjca1} sends it to the given quadratic form $q$. Here $x_{j}$ (resp.
$y_{j},z_{j})$ refers to coordinates with values $x_{j}\in\mathbb{Z}$ for
$j\in J_{1}$ resp. $x_{j}\in\{0,1,2,\ldots,n_{j}-1\}$ for $j\in J_{2}$. The
map $q\mapsto(h,c)$ is linear, so it provides a group homomorphism
$\operatorname*{Quad}(G,M)\rightarrow Z_{ab}^{3}(G,M)$, which makes Diagram
\ref{l_fig1} commute.
\end{theoremannounce}

This will be Theorem \ref{thm_QuinnFormula}. If $G$ is finite abelian (so
$J_{1}=\varnothing$ and $\#J_{2}<\infty$) and $M:=R^{\times}$, then up to
rewriting the formula for $h$ and $c$ in the multiplicative notation customary
for elements in the units $R^{\times}$, we recover precisely Quinn's formula
given in \cite[\S 2.5.2]{MR1734419}.

For some applications, especially when wanting to do explicit computations in
the setting of fusion categories over $\mathbb{C}$, the following formulation
might be more useful. We provide it with full details so that it can easily be
referenced whenever needed:

\begin{theoremannounce}
[Exponential format $3$-cocycle formula]%
\label{thmann_ExponentialFormatThreeCocycle}Suppose%
\begin{equation}
G=\bigoplus_{k\in J}\mathbb{Z}/n_{k}\mathbb{Z} \label{lviaps8x}%
\end{equation}
for $n_{k}\geq1$ and $J$ some (possibly infinite) totally ordered index set.
Write $(e_{k})_{k\in J}$ for the generator $1$ of the $k$-th summand. Then
there is a bijection between the following three sets:

\begin{enumerate}
\item All possible choices of values

\begin{itemize}
\item $p^{(k)}\in\{0,1,\ldots,\gcd(n_{k}^{2},2n_{k})-1\}$ for every $k\in J$,

\item $q^{(k,l)}\in\{0,1,\ldots,\gcd(n_{k},n_{l})-1\}$ for all $k<l$ with
$k,l\in J$.
\end{itemize}

\item All quadratic forms $q\in\operatorname*{Quad}(G,\mathbb{C}^{\times})$,
uniquely described by the following properties%
\begin{align*}
q(e_{k})  &  =\exp\left(  \frac{2\pi i}{\gcd(n_{k}^{2},2n_{k})}p^{(k)}\right)
\text{,}\\
b(e_{k},e_{l})  &  =\exp\left(  \frac{2\pi i}{\gcd(n_{k},n_{l})}%
q^{(k,l)}\right)  \qquad\text{(for }k<l\text{),}%
\end{align*}
where $b$ is the polarization of $q$ (and furthermore we necessarily then have
$b(e_{k},e_{l})=b(e_{l},e_{k})$ for $k>l$ and $b(e_{k},e_{k})=2q(e_{k})$ as well).

\item All abelian $3$-cocycles $(h,c)\in H_{ab}^{3}(G,\mathbb{C}^{\times})$,
uniquely pinned down by the cocycle representative%
\begin{align*}
c(x,y)  &  =\prod_{k<l}\exp\left(  \frac{2\pi iq^{(k,l)}}{\gcd(n_{k},n_{l}%
)}x_{k}y_{l}\right) \\
&  \qquad\cdot\prod_{k}\exp\left(  \frac{2\pi ip^{(k)}}{\gcd(2n_{k},n_{k}%
^{2})}x_{k}y_{k}\right)  \text{,}%
\end{align*}
and%
\[
h(x,y,z)=\prod_{k}\exp\left(  \frac{2\pi ip^{(k)}}{\gcd(2n_{k},n_{k}^{2}%
)}\left(  x_{k}\left(  [y_{k}]_{n_{k}}+[z_{k}]_{n_{k}}-[y_{k}+z_{k}]_{n_{k}%
}\right)  \right)  \right)  \text{,}%
\]
where $x_{k}$ (resp. $y_{k},z_{k}$) denotes the coordinates of vectors
$x,y,z\in G$ according to Equation \ref{lviaps8x}. Here $[-]_{n_{k}}$ refers
to the remainder of division by $n_{k}$, expressed as an element in
$\{0,1,\ldots,n_{k}-1\}$.
\end{enumerate}

Really, $\operatorname*{Quad}(G,\mathbb{C}^{\times})$ and $H_{ab}%
^{3}(G,\mathbb{C}^{\times})$ are abelian groups and the above bijections are
abelian group isomorphisms, given in terms of the parameters $p^{(k)}%
,q^{(k,l)}$ by elementwise addition in the quotient groups (i.e.
$\mathbb{Z}/(n_{k}^{2},2n_{k})$ for $p^{(k)}$ etc.).\newline The map
$q\mapsto(h,c)$ is linear, so it provides a group homomorphism
$\operatorname*{Quad}(G,M)\rightarrow Z_{ab}^{3}(G,M)$, which makes Diagram
\ref{l_fig1} commute.
\end{theoremannounce}

See Theorem \ref{thm_ExplicitAbelian3Cocycles}. This result appears to readily
imply various counting and enumeration problems in the literature regarding
small examples of pointed braided fusion categories for a given $G$, see
\S \ref{sect_ClassifSmallExamples}.

\subsection{Application to normal forms of associators}

In the setting of (a), i.e. pointed braided fusion categories, Quinn's formula
is sufficent to describe an associator and braiding in all situations. This is
because in this setting $G:=\pi_{0}(\mathsf{C}_{\operatorname*{simp}}%
,\otimes)$ is a finite abelian group and thus safely covered by both Theorem
\ref{thmann_GeneralizedQuinn} and Theorem
\ref{thmann_ExponentialFormatThreeCocycle}.

However, Quinn's formula has some special shape (e.g., far more symmetry than
one might a priori expect!). Let us broaden the question: Suppose that
$(\mathsf{C},\otimes)$ is a pointed braided fusion category, but drop the
assumption that there are only finitely many isomorphism classes of simple
objects. You could think of finite-dimensional $G$-graded vector spaces
$\mathsf{Vect}_{k}^{G}$ with some associator and braiding, but where the
grading comes from any abelian group $G$, and not just a finite one. We will
properly define this later and call it a \emph{big} fusion category.

Suppose we want to bring $(\mathsf{C},\otimes)$ into some particularly nice
\textquotedblleft normal form\textquotedblright\ under braided monoidal
equivalence. As before, replace $(\mathsf{C},\otimes)$ by a skeleton. Then the
associator and braiding are merely automorphisms. If $X,Y,Z$ are simple
objects, we may read $a_{X,Y,Z}$ and $s_{X,Y}$ as elements of $k^{\times}$
canonically. Now, the simplest conceivable normal form would be, through a
braided monoidal equivalence, to make all associators and the braiding
trivial. This is, however, an unrealistic hope (it cannot be achieved).
Perhaps the following is the best possible normal form one can expect in
general.\footnote{by \textquotedblleft in general\textquotedblright\ we mean:
for any abelian group. Note that, for example, if we restrict $G$ to free
abelian groups, our generalized Quinn formula directly shows that one can
always make the associator zero (because then $J_{2}=\varnothing$).}

\begin{theoremannounce}
[Extra symmetries]\label{thm_annA}Suppose $k$ is an algebraically closed field
of any characteristic. Let $(\mathsf{C},\otimes)$ be a $k$-linear pointed
braided big fusion category. Then $(\mathsf{C},\otimes)$ is braided monoidal
equivalent to a skeletal big fusion category such that%
\[
a_{X,Y,Z}=\frac{s_{X,Y}\cdot s_{X,Z}}{s_{X,Y\otimes Z}}\qquad\text{and}\qquad
a_{Z,X,Y}=\frac{s_{X\otimes Y,Z}}{s_{X,Z}\cdot s_{Y,Z}}%
\]
hold for all simple objects $X,Y,Z$.
\end{theoremannounce}

See Theorem \ref{thm_AnnFull}.\medskip

These properties are \textquotedblleft extra symmetries\textquotedblright%
\ which are not visibly forced by the hexagon and pentagon axioms. I do not
have a philosophical interpretation why such extra symmetries always exist
(e.g., note that it follows from them that $a_{X,Y,Z}=a_{X,Z,Y}$).\medskip

To restate the result in other words: The associator only measures the lack of
\textquotedblleft$\otimes$-linearity\textquotedblright\ of the braiding, in
either variable. A tool to memorize the formulas: the first argument of the
associator is the one argument which appears in all three factors on the other
side of the equality sign.\medskip

The above result follows from Quinn's formula if $(\mathsf{C},\otimes)$ is an
ordinary pointed braided fusion category with $G$ finite. We believe the above
observation is new in the case of arbitrary $G$. It does not follow by a
\textquotedblleft colimit argument\textquotedblright\ from the case of finite
$G$ by the same problem as discussed around Figure \ref{l_fig1}. And at any
rate our argument takes a different path and circumvents Quinn's formula or
its siblings.

\section{Abelian cohomology}

Let us recall Eilenberg and\ MacLane's theory of abelian cohomology. We
refrain from giving a careful motivation how this formalism arose. Instead, we
may refer to \cite[\S 3.1]{bcg} for some more background.

Let $G,M$ be abelian groups. While there are elegant and systematic
definitions of group cohomology and abelian cohomology, we will just work with
an explicit presentation here, namely normalized inhomogeneous cochains. Also,
we shall only need $H^{3}$.

Write $G^{n}:=G\times\cdots\times G$ for the $n$-fold product of abelian
groups. A \emph{group }$3$\emph{-cocycle} is a map of sets%
\[
h:G^{3}\longrightarrow M
\]
such that the identity%
\begin{equation}
h(x,y,z)+h(u,x+y,z)+h(u,x,y)=h(u,x,y+z)+h(u+x,y,z)\text{.} \label{lefmu1a}%
\end{equation}
holds for all $x,y,z\in G$ (corresponding in tensor category language to the
\textquotedblleft pentagon axiom for associators\textquotedblright, see
\S \ref{sect_AssociatorsAndBraidings}, e.g., the proof of Theorem
\ref{thm_JoyalStreetEquivBCGAndQuadTriples}).

A group $3$-cocycle is called \emph{normalized} if $h(x_{1},x_{2},x_{3})=0$ as
soon as $x_{i}=0$ for some $i\in\{1,2,3\}$. A \emph{normalized group }%
$3$\emph{-coboundary} is a group $3$-cocycle of the shape%
\begin{equation}
h(x,y,z)=k(y,z)-k(x+y,z)+k(x,y+z)-k(x,y) \label{lefmu2}%
\end{equation}
for some map of sets $k:G^{2}\rightarrow M$ such that $k(x,0)=0$ and
$k(0,y)=0$. It is easy to check that this is a normalized group $3$-cocycle.
These explicit expressions can directly be unravelled from \cite[Chapter I,
\S 2]{MR2392026} for example.

An \emph{abelian }$3$\emph{-cocycle} is a pair $(h,c)$ consisting of a group
$3$-cocycle $h:G^{3}\rightarrow M$ such that%
\begin{equation}
h(x,0,z)=0 \label{lefmu2aa}%
\end{equation}
and a map $c:G^{2}\rightarrow M$ which satisfies%
\begin{align}
h(y,z,x)+c(x,y+z)+h(x,y,z)  &  =c(x,z)+h(y,x,z)+c(x,y)\tag{A}\label{lx_a_1}\\
-h(z,x,y)+c(x+y,z)-h(x,y,z)  &  =c(x,z)-h(x,z,y)+c(y,z) \tag{A'}\label{lx_a_2}%
\end{align}
for all $x,y,z\in G$ (corresponding to the two \textquotedblleft hexagon
axioms\textquotedblright\ in the dictionary with tensor categories). Equation
\ref{lefmu2aa} implies that $h$ is normalized (\cite[Remark 3.5]{bcg}). An
\emph{abelian }$3$\emph{-coboundary}\textit{ }is a pair $(h,c)$, where $h$ is
a normalized group $3$-coboundary coming from $k:G^{2}\rightarrow M$, and%
\begin{equation}
c(x,y):=k(x,y)-k(y,x) \label{lefmu2a}%
\end{equation}
for the same $k$. Write $Z_{grp}^{3}$ (resp. $Z_{ab}^{3}$) to denote the group
of normalized group $3$-cocycles (resp. abelian $3$-cocycles), resp.
$B_{grp}^{3}$ and $B_{ab}^{3}$ for coboundaries.

\begin{definition}
We have third \emph{group cohomology}%
\[
H_{grp}^{3}(G,M)=\frac{Z_{grp}^{3}(G,M)}{B_{grp}^{3}(G,M)}=\frac
{\{\text{(normalized) group }3\text{-cocycles}\}}{\{\text{(normalized) group
}3\text{-coboundaries}\}}%
\]
and third \emph{abelian cohomology}%
\[
H_{ab}^{3}(G,M)=\frac{Z_{ab}^{3}(G,M)}{B_{ab}^{3}(G,M)}=\frac{\{\text{abelian
}3\text{-cocycles}\}}{\{\text{abelian }3\text{-coboundaries}\}}%
\]

\end{definition}

For both definitions we use normalized inhomogeneous chains, cf. \cite[Chapter
I, \S 2, Exercise 5]{MR2392026}.\medskip

By a \emph{quadratic form} $q:G\rightarrow M$ (also known as `quadratic map'
or `quadratic function' in various texts, depending on the taste of the
various authors) we mean a map of sets such that $q(x)=q(-x)$ and%
\begin{equation}
b(x,y)=q(x+y)-q(x)-q(y) \label{ljbc1}%
\end{equation}
is $\mathbb{Z}$-bilinear for all $x,y\in G$. The map $b$ is known as the
\emph{polarization form}. Write $\operatorname*{Quad}(G,M)$ for the set of all
quadratic forms. This is an abelian group under pointwise addition of maps.

\begin{example}
This definition may encompass more types of maps than a casual reader might
expect. For example if $G$ and $M$ happen to be $\mathbb{F}_{2}$-vector
spaces, every \emph{linear} map $G\rightarrow M$ is a quadratic form.
Concretely, the map%
\[
q:\mathbb{F}_{8}[X,Y]\longrightarrow\mathbb{F}_{8}[X,Y]\text{,}\qquad
q(x)=x+x^{2}+x^{4}%
\]
might not `look quadratic' as an algebraic expression, but it is a quadratic
form. Thanks to $(a+b)^{2}=a^{2}+b^{2}$ in characteristic two rings, the
polarization vanishes.
\end{example}

\begin{example}
\label{example_QuadScalar}If $q$ is any quadratic form, we have $q(nx)=n^{2}%
q(x)$ for any $x\in G$ and $n\in\mathbb{Z}$. To see this, note that%
\begin{align}
b(x,x)  &  =q(2x)-2q(x)\nonumber\\
b(x,-x)  &  =q(0)-q(x)-q(-x)=-2q(x) \label{lwaa1}%
\end{align}
both follow from Equation \ref{ljbc1}. The case $n=0$ is clear. By induction,
assuming the case $n$ to be done,%
\[
b(nx,x)=q(nx+x)-q(nx)-q(x)=q((n+1)x)-n^{2}q(x)-q(x)
\]
and by the $\mathbb{Z}$-bilinearity of $b$ and Equation \ref{lwaa1},
$b(nx,x)=-nb(x,-x)=2nq(x)$, and then%
\[
(n^{2}+2n+1)q(x)=q((n+1)x)\text{,}%
\]
proving the claim for all $n\geq0$. It follows for negative $n$ by
$q(-x)=q(x)$.
\end{example}

The key connection between abelian $3$-cocycles and quadratic forms is the
following theorem.

\begin{theorem}
[Eilenberg--Mac Lane]\label{thm_EilenbergMacLaneIso1}Let $G,M$ be abelian
groups. The so-called \emph{trace}%
\begin{align*}
\operatorname*{tr}:H_{ab}^{3}(G,M)  &  \longrightarrow\operatorname*{Quad}%
(G,M)\\
(h,c)  &  \longmapsto(x\mapsto c(x,x))
\end{align*}
is an isomorphism of abelian groups.
\end{theorem}

We give an outline how this is proven in
\S \ref{sect_ProofOfEilenbergMacLaneThm}, including a new proof of surjectivity.

\section{Admissible presentations\label{sect_AdmissiblePresentation}}

Let $G,M$ be abelian groups. Suppose $q\in\operatorname*{Quad}(G,M)$ is a
quadratic form. We write $b$ for the polarization form of $q$ as given in
Equation \ref{ljbc1}.

\begin{definition}
\label{def_AP}A \emph{pre-admissible presentation} for $q$ is a triple
$(F_{0},\pi,C)$, where

\begin{enumerate}
\item $F_{0}$ is an abelian group and $\pi$ a surjective group homomorphism%
\[
\pi:F_{0}\twoheadrightarrow G\text{;}%
\]
and write $F_{1}:=\ker(\pi)$;

\item $C$ is a $\mathbb{Z}$-bilinear form $C:F_{0}\otimes_{\mathbb{Z}}%
F_{0}\rightarrow M$ such that%
\begin{equation}
b(\pi x,\pi y)=C(x,y)+C(y,x) \label{ljbc5s}%
\end{equation}
holds for all $x,y\in F_{0}$.

\item For all $x\in F_{1}$ we have $C(x,x)=0.$
\end{enumerate}

We speak of an \emph{admissible presentation} when instead of (3) we have the
stronger property that $C(x,y)=0$ holds for all $x,y\in F_{1}$.
\end{definition}

Axiom (3) just demands that the restriction $C\mid_{F_{1}}$ is an alternating
form. For an admissible presentation, $F_{1}$ is an isotropic
subgroup.\medskip

Given a pre-admissible presentation, we can lift the quadratic form from $G$
to $F_{0}$. To this end, we define%
\begin{equation}
Q(x):=q(\pi(x))\qquad\text{for}\qquad x\in F_{0}\text{.} \label{ljbc2a}%
\end{equation}
Then $Q\in\operatorname*{Quad}(F_{0},M)$ is indeed a quadratic form. Here and
henceforth write%
\begin{equation}
B(x,y):=Q(x+y)-Q(x)-Q(y) \label{ljbc2}%
\end{equation}
for its polarization form. Note that%
\begin{equation}
B(x,y)=b(\pi x,\pi y)\text{.} \label{ljbc4}%
\end{equation}
We have chosen our notation so that the uppercase letters refer to the lifts
of their lowercase letter counterpart.

There is a slightly more refined property one can demand (and always arrange)
to hold:

\begin{definition}
We call a (pre-)admissible presentation \emph{optimal} if we have%
\begin{equation}
Q(x)=C(x,x) \label{ljbc4r}%
\end{equation}
for all $x\in F_{0}$.
\end{definition}

\begin{example}
The simplest example of a non-optimal admissible presentation is for the
quadratic form $q\in\operatorname*{Quad}(\mathbb{F}_{2},\mathbb{F}_{2})$ given
by $q(x)=x^{2}$. For this form $(\mathbb{F}_{2},\operatorname*{id}%
_{\mathbb{F}_{2}},C)$ with $C(x,y):=0$ is a non-optimal admissible
presentation with $F_{1}=0$. An optimal presentation is given by $C(x,y):=xy$.
\end{example}

Starting with any pre-admissible presentation $(F_{0},\pi,C)$, in order to
achieve optimality, one only needs to change the bilinear form $C$, while
$F_{0}$ and $\pi$ can remain the same. We prove this now.

\begin{proposition}
\label{Prop_OptimizeAP}Let $G,M$ be abelian groups and $q\in
\operatorname*{Quad}(G,M)$ a quadratic form.

\begin{enumerate}
\item If $(F_{0},\pi,C)$ is a pre-admissible presentation, one can find an
optimal pre-admissible presentation $(F_{0},\pi,C^{\prime})$.

\item If $(F_{0},\pi,C)$ is an admissible presentation, one can find an
optimal admissible presentation $(F_{0},\pi,C^{\prime})$.
\end{enumerate}
\end{proposition}

Below, we write $\left.  _{n}M\right.  :=\{m\in M\mid nm=0\}$ for the
$n$-torsion subgroup.

\begin{proof}
\textit{(Step 1) }We first prove the first claim. It is immediate to see that
$Q^{\prime}(x):=C(x,x)$ is a quadratic form on $F_{0}$. Its polarization form
is%
\[
B^{\prime}(x,y)=C(x+y,x+y)-C(x,x)-C(y,y)=C(x,y)+C(y,x)
\]
and by Equation \ref{ljbc5s} this is $b(\pi x,\pi y)$, so $Q^{\prime}$ has the
same polarization as $Q$ by Equation \ref{ljbc2}-\ref{ljbc4}. Thus,%
\begin{equation}
L:=Q^{\prime}-Q \label{lwdxs1}%
\end{equation}
is a quadratic form in $\operatorname*{Quad}(F_{0},M)$ whose polarization
vanishes. This means that $L$ satisfies%
\begin{equation}
L(x+y)-L(x)-L(y)=0 \label{lwdaa1}%
\end{equation}
for all $x,y\in F_{0}$, so $L$ is a semigroup homomorphism. Take $x=y=0$ to
obtain $L(0)=0$, and $y=-x$ to obtain $L(-x)=-L(x)$, showing that $L$ is a
morphism of abelian groups,%
\[
L:F_{0}\longrightarrow M\text{.}%
\]
As a quadratic form, it also satisfies $L(x)=L(-x)$, i.e. $2L(x)=0$ holds for
all $x\in F_{0}$. Thus, $L$ descends to a group homomorphism $L:F_{0}%
/2F_{0}\rightarrow\left.  _{2}M\right.  $. The $\mathbb{Z}$-module structure
of either side now induces an $\mathbb{F}_{2}$-vector space structure,
rendering $L$ an $\mathbb{F}_{2}$-linear map. We next choose a special basis
of $F_{0}/2F_{0}$. To this end, pick a direct sum splitting%
\begin{equation}
F_{0}/2F_{0}\simeq\operatorname*{im}(F_{1})\oplus\text{(rest),} \label{lwdxs3}%
\end{equation}
where we refer to the image coming from the inclusion $F_{1}\subseteq F_{0}$.
Pick $(\gamma_{i})_{i\in I}$ as a basis of $F_{0}/2F_{0}$ first by picking a
basis on the subspace $\operatorname*{im}(F_{1})$, say with indices in a
subset $I_{\operatorname*{im}F_{1}}\subseteq I$ of the index set, and then
prolong it to all of $F_{0}/2F_{0}$. Define a symmetric bilinear form on
$F_{0}/2F_{0}$ by%
\begin{equation}
J(x,y):=\sum_{i\in I}\overline{x}_{i}\overline{y}_{i}L(\gamma_{i})\text{,}
\label{lwdxs2}%
\end{equation}
where $\overline{x}_{i}\in\mathbb{F}_{2}$ denotes the $\mathbb{F}_{2}%
$-coordinates of the vector $x\in F_{0}/2F_{0}$ (resp. $\overline{y}_{i}$ for
$y$) with respect to the basis $(\gamma_{i})_{i\in I}$. As each $L(x)$ lies in
the $2$-torsion group $\left.  _{2}M\right.  $, the scalar multiplication with
elements from $\mathbb{F}_{2}$ is well-defined and indeed linear. For an
arbitrary $x\in F_{0}/2F_{0}$ we compute%
\[
J(x,x)=\sum_{i}\overline{x}_{i}^{2}L(\gamma_{i})\equiv\sum_{i}\overline{x}%
_{i}L(\gamma_{i})=L\left(  \sum_{i}\overline{x}_{i}\gamma_{i}\right)  =L(x)
\]
since in $\mathbb{F}_{2}$ we have $\alpha^{2}\equiv\alpha\,\left(
\operatorname{mod}2\right)  $. Under the linear surjection $F_{0}%
\twoheadrightarrow F_{0}/2F_{0}$ we can now lift $J$ to a symmetric bilinear
form%
\[
J:F_{0}\otimes_{\mathbb{Z}}F_{0}\longrightarrow\left.  _{2}M\right.  \text{.}%
\]
We keep the same name $J$ for this lift. Returning to our definition of $L$ in
Equation \ref{lwdxs1}, we now find%
\[
Q(x)=Q^{\prime}(x)-L(x)=C(x,x)-J(x,x)\text{.}%
\]
Since both $C$ and $J$ are $\mathbb{Z}$-bilinear forms, so is $C-J$. Define%
\begin{equation}
C^{\prime}:=C-J\,\text{.} \label{lwdxs5}%
\end{equation}
We thus have $Q(x)=C^{\prime}(x,x)$, so Equation \ref{ljbc4r} holds, showing
that $C^{\prime}$ is a promising candidate to satisfy the optimality
property.\newline\textit{(Step 2) }We need to check that $(F_{0},\pi
,C^{\prime})$ is a pre-admissible presentation: Axiom (1) is clear; nothing
about $F_{1}$ or $\pi$ has changed. For axiom (2) we find%
\[
C^{\prime}(x,y)+C^{\prime}(y,x)=C(x,y)-J(x,y)+C(y,x)-J(y,x)=B(x,y)-2J(x,y)
\]
since $C$ satisfies axiom (2) by assumption and $J$ is a symmetric form.
However, $J$ by construction takes values in the $2$-torsion elements $\left.
_{2}M\right.  $, so $2J(x,y)=0$ for any $x,y$. Hence, axiom (2) is satisfied.
Axiom (3): From Equation \ref{lwdxs1} and Equation \ref{ljbc2a} we get%
\[
L(x)=Q^{\prime}(x)-Q(x)=C(x,x)-q(\pi x)\text{.}%
\]
Thus, if $x\in F_{1}$ then $C(x,x)=0$ since $C$ satisfies axiom (3), and
$q(\pi x)=0$ since $F_{1}=\ker(\pi_{1})$. We deduce that $L\mid_{F_{1}}=0$.
Next, recall from Equation \ref{lwdxs2} that $J$ was defined by%
\[
J(x,y)=\sum_{i\in I}\overline{x}_{i}\overline{y}_{i}L(\gamma_{i})\text{,}%
\]
where $\overline{x}_{i},\overline{y}_{i}$ were the respective $\mathbb{F}_{2}%
$-coordinates. If $x,y\in F_{1}$, then thanks to our special choice of basis
from Equation \ref{lwdxs3}, we have $\overline{x}_{i}=0$ for all $i\in
I\setminus I_{\operatorname*{im}F_{1}}$, and the same for $\overline{y}_{i}$.
Thus, for $x,y\in F_{1}$ we have%
\[
J(x,y)=\sum_{i\in I_{\operatorname*{im}F_{1}}}\overline{x}_{i}\overline{y}%
_{i}L(\gamma_{i})=0
\]
since for $i\in I_{\operatorname*{im}F_{1}}$ we know that $\gamma_{i}$ lies in
the image of $F_{1}$ inside $F_{0}/2F_{0}$, but we had found above that
$L\mid_{F_{1}}=0$. Thus, $J(x,y)=0$ for all $x,y\in F_{1}$. Now suppose $x\in
F_{1}$. Then by Equation \ref{lwdxs5} we have%
\begin{equation}
C^{\prime}(x,x)=C(x,x)-J(x,x) \label{lwdxs4}%
\end{equation}
and this vanishes since $C(x,x)=0$ (as axiom (3)\ holds for $C$) and we had
already observed $J\mid_{F_{1}}=0$. This shows that axiom (3) holds for
$C^{\prime}$. This shows that $(F_{0},\pi,C^{\prime})$ is an optimal
pre-admissible presentation (we had shown optimality in Step 1). This finishes
the proof of the first claim. Finally, we want to show that if $(F_{0},\pi,C)$
was an admissible presentation to start with, so is $(F_{0},\pi,C^{\prime}%
)$.\ We already know the latter is optimal and pre-admissible. Now, as a
direct variation of Equation \ref{lwdxs4}, for $x,y\in F_{1}$ we get%
\[
C^{\prime}(x,y)=C(x,y)-J(x,y)
\]
and since $C$ was admissible, $C(x,y)=0$, and we had already observed
$J(x,y)=0$ above.
\end{proof}

\section{Existence theorems for admissible presentations}

We begin with the principal construction mechanism for pre-admissible
presentations. This is a good construction whenever the quadratic form comes
from a bilinear form, even if this is perhaps not possible on $G$, but only on
a bigger group.

\begin{lemmaconst}
\label{lemma_AP_QuadFormFromBilinearForm}Suppose $G$ and $M$ are arbitrary
abelian groups.

\begin{enumerate}
\item Assume one finds an abelian group $F_{0}$ with a surjection%
\[
\pi:F_{0}\twoheadrightarrow G
\]
such that on $F_{0}$ one can exhibit a bilinear form $C$ such that $q(\pi
x)=C(x,x)$ (\textquotedblleft the lift of $q$ comes from a bilinear
form\textquotedblright). Then $(F_{0},\pi,C)$ is an optimal pre-admissible presentation.

\item If one can take $F_{0}=G$ and $\pi=\operatorname*{id}_{G}$, then
$(G,\operatorname*{id}_{G},C)$ is an optimal admissible presentation and
$F_{1}=0$.
\end{enumerate}
\end{lemmaconst}

\begin{proof}
We begin with the first claim. Axiom (1) is immediate. The polarization form
$b$ of $q$, written in terms of images of elements $x,y$ from $F_{0}$, is%
\begin{align*}
b(\pi x,\pi y)  &  =q(\pi(x+y))-q(\pi x)-q(\pi y)\\
&  =C(x+y,x+y)-C(x,x)-C(y,y)=C(x,y)+C(y,x)\text{,}%
\end{align*}
proving axiom (2). Finally, if $x\in F_{1}$, then since $F_{1}=\ker(\pi)$ we
have $C(x,x)=q(\pi x)=0$, so axiom (3) holds. Optimality holds by
construction. For the second claim, note that $F_{1}=0$, so the strong form of
axiom (3) is clear.
\end{proof}

We should also mention a trivial case, where nothing much needs to be done at all.

\begin{lemmaconst}
Suppose $G$ is arbitrary and $M$ is a $\mathbb{Z}[\frac{1}{2}]$-module. Then
pick $\pi:=\operatorname*{id}_{G}$, i.e. use the presentation%
\[
0\longrightarrow G\overset{\pi}{\longrightarrow}G\longrightarrow0
\]
with%
\[
C(x,y):=\frac{1}{2}b(\pi x,\pi y)\text{.}%
\]
This is an optimal admissible presentation.
\end{lemmaconst}

\begin{proof}
Immediate. Since $b$ is symmetric, so is $C$, take $F_{0}:=G$ and then
$F_{1}=0$. In particular, there is nothing to check for axiom (3). For $x\in
G$ the polarization form yields%
\[
b(x,x)=q(2x)-q(x)-q(x)=4q(x)-2q(x)=2q(x)\text{,}%
\]
i.e. $q(x)=\frac{1}{2}b(x,x)$. Thus, for $x\in F_{0}$ we have $C(x,x)=q(\pi
x)$, so Equation \ref{ljbc4r} holds.
\end{proof}

The hypothesis of $M$ being a $\mathbb{Z}[\frac{1}{2}]$-module is virtually
never satisfied in applications though.

\begin{example}
Suppose $(\mathsf{C},\otimes)$ is any $k$-linear braided fusion category. Then
the group which appears for $M$ in applications (see
\S \ref{sect_AssociatorsAndBraidings}) is $M:=\pi_{1}(\mathsf{C}%
,\otimes)=k^{\times}$ since the tensor unit $1_{\mathsf{C}}$ is a simple
object. For this group to be $2$-divisible one needs that $k$ is closed under
all quadratic field extensions. This is for example satisfied if $k$ is
algebraically closed. However, to be free from $2$-torsion one also needs that
$x^{2}=1$ implies $x=1$, i.e. $+1=-1$. This forces $k$ to be of characteristic two.
\end{example}

Instead, a much more realistic hypothesis is that $M$ is a divisible module.
If $k$ is any algebraically closed field, $k^{\times}$ is a divisible group.

\begin{lemma}
\label{lemma_DivisibleAndFreeCaseMakePreAdmAdmissible}Suppose $G$ is an
arbitrary abelian group and $M$ a divisible abelian group. If $(F_{0},\pi,C)$
is a pre-admissible presentation with $F_{0}$ free abelian, then one can
replace $C$ by a bilinear form $\tilde{C}$ such that $(F_{0},\pi,\tilde{C})$
is an admissible presentation. If the presentation was optimal to start with,
the new $\tilde{C}$ can be taken optimal, too.
\end{lemma}

\begin{proof}
Consider the restriction $C\mid_{F_{1}}$. We have $C(x,x)=0$ for all $x\in
F_{1}$ by axiom (3). Thus, $C\mid_{F_{1}}$ is an alternating bilinear
form\footnote{just to be sure about nomenclature: \emph{Alternating} means
that $C(x,x)=0$ for all $x$. This implies $C(x,y)=-C(y,x)$, but is strictly
more restrictive than the latter property.},%
\[
\left.  C\mid_{F_{1}}\right.  \in\operatorname*{Hom}\nolimits_{\mathbb{Z}%
}(\operatorname*{Alt}\nolimits^{2}(F_{1}),M)\text{.}%
\]
Since $F_{0}$ is free, so is $F_{1}$ ($\mathbb{Z}$ is a hereditary ring). Then
the injectivity $F_{1}\hookrightarrow F_{0}$ implies that the top horizontal
arrow in the diagram%
\begin{equation}%
\xymatrix{
\operatorname{Alt}\nolimits^{2}(F_{1}) \ar@{^{(}->}[r] \ar[dr]_{{C\mid_{F_1}}}
& \operatorname{Alt}\nolimits^{2}(F_{0}) \ar@{-->}[d]^{A} \\
& M
}
\label{lwwa5p}%
\end{equation}
is also injective (without too much harm it can be checked directly that
exterior powers of free modules preserve injectivity, but a literature
reference would be \cite[Theorem 1]{MR212044}). Since $M$ is divisible, it is
injective as a $\mathbb{Z}$-module, so the dashed arrow $A$ above exists and
makes the diagram commute. Define%
\[
\tilde{C}(x,y):=C(x,y)-A(x,y)\text{.}%
\]
We claim that $(F_{0},\pi,\tilde{C})$ is an admissible presentation. We
compute for $x,y\in F_{0}$ that%
\[
\tilde{C}(x,y)+\tilde{C}(y,x)=C(x,y)+C(y,x)-A(x,y)-A(y,x)=B(x,y)+0
\]
since $A$ is alternating, so Equation \ref{ljbc5s} is still valid. This
settles axioms (1) and (2). Moreover, for $x,y\in F_{1}$ we have%
\[
\tilde{C}(x,y)=C(x,y)-A(x,y)=0
\]
since $C\mid_{F_{1}}=A\mid_{F_{1}}$ by Diagram \ref{lwwa5p}, so axiom (3)
holds and we really have an admissible presentation. Regarding optimality,
note that $A(x,x)=0$, so $\tilde{C}(x,x)=C(x,x)$.
\end{proof}

The next construction will be the concrete input needed for establishing a
generalized form of Quinn's formula.

\begin{lemmaconst}
\label{lemma_ConstructAP}Suppose $M$ is arbitrary. If $G$ is a (possibly
infinite) direct sum of various (possibly infinite) cyclic groups, an optimal
admissible presentation for $q$ exists. (A concrete construction is given in
the proof)
\end{lemmaconst}

\begin{proof}
[Proof and Construction]\textit{(Step 1)} A cyclic group is either of the form
$\mathbb{Z}$ or $\mathbb{Z}/n$ for some $n\geq1$. Thus, each direct summand in
$G$ has a presentation of the shape%
\[
0\rightarrow\mathbb{Z}\overset{\cdot n}{\rightarrow}\mathbb{Z}\rightarrow
\mathbb{Z}/n\rightarrow0\qquad\text{or}\qquad0\rightarrow0\rightarrow
\mathbb{Z}\rightarrow\mathbb{Z}\rightarrow0\text{.}%
\]
Take the direct sum of these, i.e. we have found a presentation%
\begin{equation}
0\longrightarrow\underset{i\in I}{\bigoplus}\,\mathbb{Z}\longrightarrow
\underset{j\in J}{\bigoplus}\,\mathbb{Z}\overset{\pi}{\longrightarrow
}G\longrightarrow0 \label{ljbc4a}%
\end{equation}
for suitable index sets $I\subseteq J$, where the first arrow is given by a
diagonal matrix. We take this as $F_{1}\hookrightarrow F_{0}\twoheadrightarrow
G$. This sets up (1) in an admissible presentation. Use the same notation $Q$
and $B$ as in Equation \ref{ljbc2a} and \ref{ljbc2} now.\newline\textit{(Step
2)} Write $(e_{j})_{j\in J}$ for the basis vectors of $F_{0}$. Fix a total
order on $J$. Define%
\begin{equation}
C(e_{i},e_{j}):=\left\{
\begin{array}
[c]{ll}%
B(e_{i},e_{j}) & \text{if }i<j\\
Q(e_{i}) & \text{if }i=j\\
0 & \text{if }i>j
\end{array}
\right.  \label{lwbxy1}%
\end{equation}
on this basis. Prolong it uniquely to all of $F_{0}$ by $\mathbb{Z}%
$-bilinearity. For any $i,j\in J$ with $i\neq j$ we find%
\[
C(e_{i},e_{j})+C(e_{j},e_{i})=B(e_{i},e_{j})
\]
since $b$ is symmetric. If $i=j$ then $C(e_{i},e_{j})+C(e_{j},e_{i}%
)=2Q(e_{i})$, while we also have%
\[
B(e_{i},e_{i})=Q(2e_{i})-2Q(e_{i})=2Q(e_{i})
\]
by using Equation \ref{ljbc2} (and Example \ref{example_QuadScalar}). For the
polarization $B^{\prime}$ of the quadratic form $Q^{\prime}(x):=C(x,x)$ we
compute%
\[
B^{\prime}(x,y)=C(x+y,x+y)-C(x,x)-C(y,y)=C(x,y)+C(y,x)
\]
and thus%
\[
B^{\prime}(e_{i},e_{j})=\left\{
\begin{array}
[c]{ll}%
B(e_{i},e_{j}) & \text{if }i\neq j\\
2Q(e_{i}) & \text{if }i=j\text{.}%
\end{array}
\right.
\]
We find that the polarization forms of $Q$ and $Q^{\prime}$ agree. Thus,
$L:=Q-Q^{\prime}$ is a quadratic form whose polarization is zero. With the
same computation as in Equation \ref{lwdaa1} it follows that $L:F_{0}%
/2F_{0}\rightarrow\left.  _{2}M\right.  $ is a group homomorphism. We compute%
\[
L(e_{i})=Q(e_{i})-Q^{\prime}(e_{i})=Q(e_{i})-C(e_{i},e_{i})=0\text{,}%
\]
just by unravelling the definition of $Q^{\prime}$ and using Equation
\ref{lwbxy1}. Since the $(e_{j})_{j\in J}$ form a basis of $F_{0}$ and $L$ is
linear, it follows that $L(x)=0$ for all $x\in F_{0}$. Thus, $Q(x)=Q^{\prime
}(x)=C(x,x)$ for all $x\in F_{0}$. This shows that $Q$ comes from a bilinear
form, so we may invoke Lemma \ref{lemma_AP_QuadFormFromBilinearForm} and learn
that $(F_{0},\pi,C)$ is an optimal pre-admissible presentation.\newline%
\textit{(Step 3)} Finally, we need to show the strong form of axiom (3). Note
that by our special construction of the presentation in Equation \ref{ljbc4a},
the group $F_{1}=\ker(\pi)$ is generated by elements of the shape $(n_{i}%
e_{i})_{i\in I}$ with $n_{i}\in\mathbb{Z}_{\geq1}$, i.e. $\pi(n_{i}e_{i})=0$.
We compute for arbitrary $i,j\in I$ (again using the fact from Example
\ref{example_QuadScalar}),%
\begin{align*}
C(n_{i}e_{i},n_{j}e_{j})  &  =n_{i}n_{j}C(e_{i},e_{j})\\
&  =\left\{
\begin{array}
[c]{ll}%
n_{i}n_{j}B(e_{i},e_{j})=B(n_{i}e_{i},n_{j}e_{j}) & \text{if }i<j\\
n_{i}^{2}Q(e_{i})=Q(n_{i}e_{i}) & \text{if }i=j\\
0 & \text{if }i>j\text{,}%
\end{array}
\right.
\end{align*}
but of course by the very definition of $Q$ and $B$, these terms all vanish
since we have $\pi(n_{i}e_{i})=0$. This proves that $(F_{0},\pi,C)$ is admissible.
\end{proof}

Finally, let us show that optimal (pre-)admissible presentations always exist
under very general hypotheses. However, the constructions in the proof cannot
be carried out constructively usually, so the following theorem will not help
when wanting to develop explicit formulas.

\begin{theorem}
[Abstract Existence]\label{thm_AbstractExistence}Suppose $G,M$ are abelian groups.

\begin{enumerate}
\item Then for any quadratic form $q:G\rightarrow M$ an optimal pre-admissible
presentation exists.

\item If $M$ is divisible, then for any quadratic form $q:G\rightarrow M$ an
optimal admissible presentation exists.
\end{enumerate}
\end{theorem}

\begin{proof}
\textit{(Step 1)} We imitate the method of Lemma \ref{lemma_ConstructAP} for
as long as possible. First, pick a free resolution of $G$,%
\[
0\longrightarrow\underset{i\in I}{\bigoplus}\,\mathbb{Z}\longrightarrow
\underset{j\in J}{\bigoplus}\,\mathbb{Z}\overset{\pi}{\longrightarrow
}G\longrightarrow0
\]
for suitable index sets $I,J$. This always exists (and has length $2$ either
because $\mathbb{Z}$ is a ring of global dimension one, or, more down to
earth, since the kernel of $\pi$ has to be free abelian itself). As before,
write $(e_{j})_{j\in J}$ for the basis vectors of $F_{0}$, fix a total order
on $J$, and let $Q$ and $B$ denote the lifts of $q$ and $b$ to $F_{0}$ (as in
Equations \ref{ljbc2a}-\ref{ljbc2}).\ This replaces Step 1 in the proof of
Lemma \ref{lemma_ConstructAP}.\newline\textit{(Step 2)} Define%
\begin{equation}
C(e_{i},e_{j}):=\left\{
\begin{array}
[c]{ll}%
B(e_{i},e_{j}) & \text{if }i<j\\
Q(e_{i}) & \text{if }i=j\\
0 & \text{if }i>j\text{.}%
\end{array}
\right.  \label{lwbxa1}%
\end{equation}
Now repeat the same arguments as in Step 2 of the proof of Lemma
\ref{lemma_ConstructAP}. This all goes through and proves that $(F_{0},\pi,C)$
is an optimal pre-admissible presentation. This proves the first claim. Step 3
in the cited proof does not adapt to the present setting. We do something
else:\newline\textit{(Step 3)} If $M$ is divisible, we can invoke Lemma
\ref{lemma_DivisibleAndFreeCaseMakePreAdmAdmissible} and transform the
construction from Step 2 into an admissible optimal presentation $(F_{0}%
,\pi,\tilde{C})$. This settles the second claim.
\end{proof}

\begin{problem}
Does any quadratic form $q:G\rightarrow M$ admit an optimal admissible
presentation without assuming $M$ divisible?
\end{problem}

Thanks to Proposition \ref{Prop_OptimizeAP} one would only need to exhibit an
admissible presentation; the optimality can be achieved afterwards.

\begin{example}
We illustrate that Step 1 in the above proof cannot be expected to give
admissible presentations right away. Consider the needlessly complicated free
resolution%
\[
\mathbb{Z}^{n-1}\longrightarrow\mathbb{Z}^{n}\overset{\pi}{\longrightarrow
}\mathbb{Z}\text{,}%
\]
where $\pi(x_{1},\ldots,x_{n})=\sum_{i=1}^{n}x_{i}$. For the quadratic form
$x\mapsto x^{2}$ on $\mathbb{Z}$, the procedure in the proof of Theorem
\ref{thm_AbstractExistence} yields that $(\mathbb{Z}^{n},\pi,C)$ with%
\[
C(e_{i},e_{j})=\left\{
\begin{array}
[c]{ll}%
2 & \text{if }i<j\\
1 & \text{if }i=j\\
0 & \text{if }i>j
\end{array}
\right.
\]
is an optimal pre-admissible presentation. All the vectors $e_{i}-e_{k}$ lie
in the kernel $F_{1}=\ker(\pi)$. For $i<j<k$ we compute $C(e_{i}-e_{k}%
,e_{j}-e_{k})=1$.
\end{example}

\section{The lifting function\label{sect_Lift}}

Suppose $q\in\operatorname*{Quad}(G,M)$ is a quadratic form and assume we have
chosen an admissible presentation $(F_{0},\pi,C)$ as in
\S \ref{sect_AdmissiblePresentation}.

\begin{definition}
\label{def_AL}For any non-zero element $x\in G$ pick once and for all a lift
$\widetilde{x}\in F_{0}$, i.e. some element such that $\pi(\widetilde{x})=x$.
For the neutral element we pick the special lift%
\begin{equation}
\widetilde{0}:=0\text{.} \label{ljbc5a}%
\end{equation}
Call any such choice an \emph{admissible lift}.
\end{definition}

As $\pi$ is surjective, it is clear that admissible lifts always exist.

\begin{example}
We stress that we have $\pi\widetilde{x}=x$, but in general there is not much
we can say about how $\widetilde{(-)}$ interacts with algebraic operations.
For example, $\widetilde{2x}\neq2\widetilde{x}$, $\widetilde{(-x)}%
\neq-\widetilde{(x)}$ or $\widetilde{x+y}\neq\widetilde{x}+\widetilde{y}$ are
all possible in suitably chosen examples, and in general $\pi$ will not admit
a splitting in terms of abelian groups, so in general we cannot avoid for
these lifts to depend somewhat non-linearly on the input.
\end{example}

Having fixed an admissible lift, define a map%
\[
L:G\times G\longrightarrow M
\]%
\begin{equation}
L(x,y):=\widetilde{(x+y)}-\widetilde{x}-\widetilde{y}\text{.} \label{lbcg1a}%
\end{equation}
Note that there is \textit{no reason} why $L$ would have to be bilinear in any
way. We can record a few useful facts about $L$ nonetheless:

\begin{lemma}
\label{lemma_LSymmetries}Fix admissible lifts and suppose $(F_{0},\pi,C)$ is
an admissible presentation. We have

\begin{enumerate}
\item $L(0,y)=0$,

\item $L(x,y)=L(y,x)$,

\item $L(x+y,z)-L(x,y+z)=L(y,z)-L(x,y)$,

\item $C(L(u,x),L(y,z))=0$.
\end{enumerate}

for all $u,x,y,z\in G$.
\end{lemma}

For merely pre-admissible presentations property (4) can fail.

\begin{proof}
(1) We find%
\[
L(0,y)=\widetilde{y}-\widetilde{0}-\widetilde{y}=0
\]
using our special choice of lift in Equation \ref{ljbc5a}. (2) Obvious. (3) We
find%
\begin{align*}
L(x+y,z)-L(x,y+z)  &  =\widetilde{(x+y+z)}-\widetilde{(x+y)}-\widetilde
{z}-\widetilde{(x+y+z)}+\widetilde{x}+\widetilde{(y+z)}\\
&  =\widetilde{x}-\widetilde{(x+y)}+\widetilde{(y+z)}-\widetilde
{z}=L(y,z)-L(x,y)\text{,}%
\end{align*}
where we have just used cancellations of terms. (4) Note that for all $x,y\in
G$ we have $\pi(\widetilde{(x+y)}-\widetilde{x}-\widetilde{y})=0$, so
$\widetilde{(x+y)}-\widetilde{x}-\widetilde{y}\in F_{1}$. This proves (4)
since this applies to both arguments, so we can use the strong form of axiom
(3) of an admissible presentation.
\end{proof}

\section{\label{sect_ConstructCocycle}Constructing abelian $3$-cocycles}

Let $q\in\operatorname*{Quad}(G,M)$ be a quadratic form. Suppose we have fixed
an admissible presentation $(F_{0},\pi,C)$ as in Definition \ref{def_AP},
alongside a choice of admissible lifts as in Definition \ref{def_AL}. We use
the notation $F_{0},F_{1},\pi,C,B,Q$ as explained in
\S \ref{sect_AdmissiblePresentation}.

Define maps%
\[
h:G\times G\times G\longrightarrow M\qquad\text{and}\qquad c:G\times
G\longrightarrow M
\]
by%
\begin{equation}
h(x,y,z):=-C(\widetilde{x},L(y,z))\qquad\text{and}\qquad c(x,y):=C(\widetilde
{x},\widetilde{y})\text{,} \label{ldefio1}%
\end{equation}
where $C$ is the bilinear form of the admissible presentation, $\widetilde
{(-)}$ denotes the admissible lift and $L$ is the (non-linear!) pairing of
Equation \ref{lbcg1a}.

Again, note that there is no reason why $h$ or $c$ would be multilinear.

\begin{lemma}
[Key Lemma]\label{lemma_E2}The datum $(h,c)$ of Equation \ref{ldefio1}
describes an abelian $3$-cocycle.
\end{lemma}

We stress that we only need the above assumptions, i.e. the admissible
presentation $(F_{0},\pi,C)$ does not need to be optimal.

\begin{proof}
\textit{(Step 0) }We have $h(x,0,z)=-C(\widetilde{x},L(0,z))=0$ by Lemma
\ref{lemma_LSymmetries} and since $C$ is $\mathbb{Z}$-bilinear. \textit{(Step
1)} We first check Equation \ref{lx_a_1}. We unravel%
\begin{align*}
h(y,z,x)+h(x,y,z)-h(y,x,z)  &  =-C(\widetilde{y},L(z,x))-C(\widetilde
{x},L(y,z))+C(\widetilde{y},L(x,z))\\
&  =-C(\widetilde{x},L(y,z))
\end{align*}
as the first and third term cancel each other out since $L$ is symmetric
(Lemma \ref{lemma_LSymmetries}). Now unpack the definition of $L$ and use that
$C$ is $\mathbb{Z}$-bilinear on $F_{0}$, giving%
\[
=-C(\widetilde{x},\widetilde{y+z})+C(\widetilde{x},\widetilde{y}%
)+C(\widetilde{x},\widetilde{z})=-c(x,y+z)+c(x,y)+c(x,z)\text{,}%
\]
which confirms Equation \ref{lx_a_1}. \textit{(Step 2) }Next, we check
Equation \ref{lx_a_2}, which is a little asymmetric in comparison to the
previous computation: We unravel%
\begin{align*}
h(x,z,y)-h(z,x,y)-h(x,y,z)  &  =-C(\widetilde{x},L(z,y))+C(\widetilde
{z},L(x,y))+C(\widetilde{x},L(y,z))\\
&  =C(\widetilde{z},L(x,y))\text{,}%
\end{align*}
again using that $L$ is symmetric. Again, unpack $L$ and use the bilinearity
of $C$, giving%
\[
=C(\widetilde{z},\widetilde{x+y})-C(\widetilde{z},\widetilde{x})-C(\widetilde
{z},\widetilde{y})\text{.}%
\]
Next, by Equation \ref{ljbc5s} we have $C(y,x)=B(x,y)-C(x,y)$ for all $x,y\in
F_{0}$. Thus, rewrite the preceding equation as%
\begin{align}
&  =B(\widetilde{x+y},\widetilde{z})-B(\widetilde{x},\widetilde{z}%
)-B(\widetilde{y},\widetilde{z})\label{ljbc5}\\
&  -C(\widetilde{x+y},\widetilde{z})+C(\widetilde{x},\widetilde{z}%
)+C(\widetilde{y},\widetilde{z})\text{.}\nonumber
\end{align}
However, we also have Equation \ref{ljbc4}, namely $B(x,y)=b(\pi x,\pi y)$, so
the first line simplifies to%
\[
b(\pi(\widetilde{x+y}),\pi(\widetilde{z}))-b(\pi\widetilde{x},\pi\widetilde
{z})-b(\pi\widetilde{y},\pi\widetilde{z})
\]
but $\pi\widetilde{x}=x$ for all $x\in G$, so this equals
$b(x+y,z)-b(x,z)-b(y,z)$. Since $b$ is the polarization form of $q$, Equation
\ref{ljbc1}, $b$ is $\mathbb{Z}$-bilinear, so this expression vanishes for all
$x,y,z\in G$. Thus, Equation \ref{ljbc5} simplifies to%
\[
=-C(\widetilde{x+y},\widetilde{z})+C(\widetilde{x},\widetilde{z}%
)+C(\widetilde{y},\widetilde{z})=-c(x+y,z)+c(x,z)+c(y,z)\text{,}%
\]
which confirms Equation \ref{lx_a_2}.\textit{ (Step 3) }Finally, we need to
check whether $h$ is a group $3$-cocycle, i.e. confirm whether%
\begin{equation}
h(x,y,z)+h(u,x+y,z)+h(u,x,y)-h(u,x,y+z)-h(u+x,y,z)=0 \label{ljbc5b}%
\end{equation}
holds for all $x,y,z,u\in G$. We first evaluate%
\[
h(u,x+y,z)=-C(\widetilde{u},L(x+y,z))
\]
and relying on $L(x+y,z)=L(x,y+z)+L(y,z)-L(x,y)$ (an equality stemming from
Lemma \ref{lemma_LSymmetries}, (3), rearranged), the preceding equation can be
rewritten as%
\begin{align*}
h(u,x+y,z)  &  =-C(\widetilde{u},L(x,y+z)+L(y,z)-L(x,y))\\
&  =-C(\widetilde{u},L(x,y+z))-C(\widetilde{u},L(y,z))+C(\widetilde
{u},L(x,y))\\
&  =h(u,x,y+z)+h(u,y,z)-h(u,x,y)\text{,}%
\end{align*}
where we have used that $C$ is a $\mathbb{Z}$-bilinear form on $F_{0}$. Plug
this into the left-hand side of Equation \ref{ljbc5b}, showing that it
suffices to prove%
\[
h(u,y,z)+h(x,y,z)-h(u+x,y,z)=0\text{.}%
\]
However, unravelling the definition of $h$, this simplifies to%
\[
=-C(\widetilde{u}+\widetilde{x}-\widetilde{(u+x)},L(y,z))=C(L(u,x),L(y,z))
\]
since $C$ is $\mathbb{Z}$-bilinear. This proves the desired vanishing by using
the last property shown in Lemma \ref{lemma_LSymmetries}.
\end{proof}

Finally, we are ready for our main result.

\begin{theorem}
\label{thm_GeneralFormulaWithAP}Let $G,M$ be abelian groups and $q\in
\operatorname*{Quad}(G,M)$. Suppose

\begin{itemize}
\item $(F_{0},\pi,C)$ is an optimal admissible presentation, and

\item $\widetilde{(-)}$ is an admissible lifting.
\end{itemize}

Then%
\[
h(x,y,z):=-C(\widetilde{x},L(y,z))\text{,}\qquad c(x,y):=C(\widetilde
{x},\widetilde{y})
\]
with the non-linear function $L(x,y):=\widetilde{(x+y)}-\widetilde
{x}-\widetilde{y}$, defines an abelian $3$-cocycle whose attached quadratic
form is $q$, i.e.%
\[
H_{ab}^{3}(G,M)\longrightarrow\operatorname*{Quad}(G,M)\text{,}\qquad
\qquad(h,c)\longmapsto q
\]
under the map of Theorem \ref{thm_EilenbergMacLaneIso1}.
\end{theorem}

Recall that if one only is given a non-optimal admissible presentation, one
can always change it into an optimal one by Proposition \ref{Prop_OptimizeAP}.

\begin{proof}
This is easy now. Firstly, by Lemma \ref{lemma_E2} $(h,c)$ is an abelian
$3$-cocycle. The trace maps it to the quadratic form%
\[
q^{\prime}(x)\underset{(1)}{=}c(x,x)\underset{(2)}{=}C(\widetilde
{x},\widetilde{y})\underset{(3)}{=}Q(\widetilde{x})=q(\pi\widetilde{x})=q(x)
\]
for $x\in G$. Here (1) is just the definition of the trace map from abelian
$3$-cocycles to quadratic forms, (2) is the definition of $c(-,-)$, (3) is the
optimality of the admissible presentation (Equation \ref{ljbc4r}), and the
rest unravels definitions.
\end{proof}

\section{Generalized Quinn formula}

We can now use the tools of \S \ref{sect_ConstructCocycle} to reprove Quinn's
formula in a generalized format. In particular, this gives an alternative
approach to the original proof for $G$ finite abelian \cite[\S 2.5.1-2.5.2]%
{MR1734419}.

Below, we intentionally stay close to the notation of Quinn's article so that
the resulting formula has the same shape.

\begin{theorem}
[Generalized\ Quinn formula]\label{thm_QuinnFormula}Let $M$ be any abelian
group. Suppose%
\begin{equation}
G=\left(  \bigoplus_{j\in J_{1}}\mathbb{Z}\right)  \oplus\left(
\bigoplus_{j\in J_{2}}\mathbb{Z}/n_{j}\mathbb{Z}\right)  \label{lquinn1}%
\end{equation}
for $J_{1},J_{2}$ any index sets, and $n_{j}\geq1$ suitable integers. Fix a
total order on the disjoint union $J:=J_{1}\dot{\cup}J_{2}$, say with
$J_{1}<J_{2}$. Write $(\mathsf{e}_{j})_{j\in J}$ for the standard generators
(i.e. the element $1_{\mathbb{Z}}$ resp. $1_{\mathbb{Z}/n_{j}\mathbb{Z}}$ in
the corresponding summand). Let $q\in\operatorname*{Quad}(G,M)$ be a quadratic
form and%
\begin{equation}
b(x,y):=q(x+y)-q(x)-q(y) \label{lzwkx1}%
\end{equation}
its polarization. Define%
\[
\sigma_{i,j}:=\left\{
\begin{array}
[c]{ll}%
b(\mathsf{e}_{i},\mathsf{e}_{j}) & \text{if }i<j\\
q(\mathsf{e}_{i}) & \text{if }i=j\\
0 & \text{if }i>j\text{.}%
\end{array}
\right.
\]
Then the pair $(h,c)$ with%
\[
h(x,y,z):=\sum_{\substack{j\in J_{2}\\\operatorname*{with}\text{ }y_{j}%
+z_{j}\geq n_{j}}}x_{j}n_{j}\sigma_{j,j}\qquad\text{and}\qquad c(x,y):=\sum
_{\substack{i,j\in J\\\operatorname*{with}\text{ }i\leq j}}x_{i}y_{j}%
\sigma_{i,j}%
\]
defines an abelian $3$-cocycle such that the trace map of Equation
\ref{lbjca1} sends it to the given quadratic form $q$. Here $x_{j}$ (resp.
$y_{j},z_{j})$ refers to coordinates with values $x_{j}\in\mathbb{Z}$ for
$j\in J_{1}$ resp. $x_{j}\in\{0,1,2,\ldots,n_{j}-1\}$ for $j\in J_{2}$. The
map $q\mapsto(h,c)$ is linear, so it provides a group homomorphism
$\operatorname*{Quad}(G,M)\rightarrow Z_{ab}^{3}(G,M)$, which makes Diagram
\ref{l_fig1} commute.
\end{theorem}

\begin{proof}
[Proof of Theorem \ref{thm_QuinnFormula}]\textit{(Step 1)} Given the format of
our input abelian group $G$, we can use Lemma/Construction
\ref{lemma_ConstructAP} to set up an optimal admissible presentation. Let us
quickly walk through the relevant steps of the construction, adapted to our
setting:\ Define%
\[
F_{1}:=\bigoplus_{j\in J_{2}}\,\mathbb{Z}\qquad\text{and}\qquad F_{0}%
:=\bigoplus_{j\in J_{1}\cup J_{2}}\,\mathbb{Z}%
\]
(with $J_{1}$ and $J_{2}$ as in Equation \ref{lquinn1}) and this sets up a
resolution%
\[
0\longrightarrow F_{1}\longrightarrow F_{0}\overset{\pi}{\longrightarrow
}G\longrightarrow0\text{.}%
\]
Write $(e_{j})_{j\in J}$ for the standard basis of $F_{0}$, i.e. the $j$-th
summand $\mathbb{Z}$ is spanned by $e_{j}$ (so that $\mathsf{e}_{j}=\pi e_{j}$
in terms of the elements in the statement of the theorem). As in Lemma
\ref{lemma_ConstructAP}, define $B(x,y)=b(\pi x,\pi y)$ and $Q(x)=q(\pi x)$
and then%
\begin{equation}
\sigma_{i,j}:=C(e_{i},e_{j})=\left\{
\begin{array}
[c]{ll}%
B(e_{i},e_{j}) & \text{if }i<j\\
Q(e_{i}) & \text{if }i=j\\
0 & \text{if }i>j
\end{array}
\right.  \label{lziv3}%
\end{equation}
describes a $\mathbb{Z}$-bilinear form on $F_{0}$ (this is the same as in the
construction given loc. cit.). As guaranteed by the quoted lemma, $(F_{0}%
,\pi,C)$ is an optimal admissible presentation. \textit{(Step 2)} Each element
of $G$ has a unique presentation as%
\[
g=\sum_{j\in J}g_{j}\pi(e_{j})\qquad\text{with}\qquad g_{j}\in\{0,1,\ldots
,n_{j}-1\}\text{ if }j\in J_{2}%
\]
and $g_{j}\in\mathbb{Z}$ if $j\in J_{1}$. Sending this $g$ to the vector%
\[
\widetilde{g}:=\sum_{j\in J}g_{j}e_{j}\in F_{0}%
\]
pins down an admissible lift in the sense of Definition \ref{def_AL}
(including $\widetilde{0}=0$). With respect to the basis $(e_{j})_{j\in J}$ we
can write $C(-,-)$ as%
\begin{equation}
C(x,y)=\sum_{i,j\in J}x_{i}y_{j}C(e_{i},e_{j})=\sum_{i\leq j}x_{i}y_{j}%
\sigma_{i,j}\text{.} \label{lziv3mx}%
\end{equation}
We can write the admissible lift coordinate-wise for any vector $x\in G$ as%
\begin{equation}
(\widetilde{x})_{j}=\left\{
\begin{array}
[c]{ll}%
x_{j} & \text{for }j\in J_{1}\\
x_{j}-n_{j}\left\lfloor \frac{x_{j}}{n_{j}}\right\rfloor  & \text{for }j\in
J_{2}\text{.}%
\end{array}
\right.  \label{lviaps1}%
\end{equation}
In particular,%
\begin{align*}
L(x,y)_{j}  &  =(\widetilde{x+y})_{j}-(\widetilde{x})_{j}-(\widetilde{y}%
)_{j}\\
&  =\left\{
\begin{array}
[c]{ll}%
0 & \text{for }j\in J_{1}\\
-n_{j}\left(  \left\lfloor \frac{x_{j}+y_{j}}{n_{j}}\right\rfloor
-\left\lfloor \frac{x_{j}}{n_{j}}\right\rfloor -\left\lfloor \frac{y_{j}%
}{n_{j}}\right\rfloor \right)  & \text{for }j\in J_{2}\text{.}%
\end{array}
\right.
\end{align*}
Now assume we are given $x,y,z\in F_{0}$ such that the coordinates satisfy the
bound $x_{j}\in\{0,1,\ldots,n_{j}-1\}$ for all $j\in J_{2}$ (this will
simplify the formulas). Demand the same for $y_{j}$ resp. $z_{j}$. There is no
condition if $j\in J_{1}$. Invoke Theorem\ \ref{thm_GeneralFormulaWithAP} to
obtain that the pair $(h,c)$ with%
\[
h(x,y,z):=-C(\widetilde{x},L(y,z))\qquad\text{and}\qquad c(x,y):=C(\widetilde
{x},\widetilde{y})
\]
is an abelian $3$-cocycle mapping to $q$ under the trace map. Next, let us
unravel these expressions. Expand $h$ using Equation \ref{lziv3mx} to%
\begin{align}
h(x,y,z)  &  =-\sum_{i\leq j}x_{i}L(y,z)_{j}\sigma_{i,j}\label{lviaps3}\\
&  =\sum_{i\leq j\text{ with }j\in J_{2}}x_{i}\left(  \left\lfloor \frac
{y_{j}+z_{j}}{n_{j}}\right\rfloor -\left\lfloor \frac{y_{j}}{n_{j}%
}\right\rfloor -\left\lfloor \frac{z_{j}}{n_{j}}\right\rfloor \right)
n_{j}\sigma_{i,j}\text{.} \label{lviaps10a}%
\end{align}
Since $x_{j},y_{j}\in\{0,1,\ldots,n_{j}-1\}$ for every $j\in J_{2}$, we may
rewrite the expression for $h$ as%
\begin{equation}
h(x,y,z)=\sum_{i\leq j\text{ with }j\in J_{2}}x_{i}\left\{
\begin{array}
[c]{cc}%
0 & \text{if }y_{j}+z_{j}<n_{j}\\
1 & \text{if }y_{j}+z_{j}\geq n_{j}%
\end{array}
\right\}  n_{j}\sigma_{i,j} \label{lchiav1}%
\end{equation}
and this simplifies to%
\[
=\sum_{\substack{i\leq j\\\text{with }j\in J_{2}\text{ and }y_{j}+z_{j}\geq
n_{j}}}x_{i}n_{j}\sigma_{i,j}\text{.}%
\]
Finally, if $i\neq j$, we have by Equation \ref{lziv3} and the bilinearity of
$B$ that%
\begin{equation}
n_{j}\sigma_{i,j}=n_{j}B(e_{i},e_{j})=B(e_{i},n_{j}e_{j})=b(\pi e_{i}%
,\pi(n_{j}e_{j}))=0 \label{lviaps4}%
\end{equation}
since $n_{j}e_{j}\in\ker(\pi)$. Hence,%
\[
h(x,y,z)=\sum_{\substack{j\in J_{2}\\\text{with }y_{j}+z_{j}\geq n_{j}}%
}x_{j}n_{j}\sigma_{j,j}\text{.}%
\]
This finishes the proof.
\end{proof}

\begin{example}
Note that along the way, we have found some other possibly useful
presentations of the $3$-cocycle. For example, Equation \ref{lviaps10a}
expresses the associator for arbitrary representatives/lifts $x_{i}%
,y_{i},z_{i}\in\mathbb{Z}$.
\end{example}

\begin{example}
[\cite{MR2755172}]A lively description how one attaches an abelian $3$-cocycle
to a quadratic form is also given by Kapustin and Saulina in \cite[\S 3.2]%
{MR2755172} (again in the situation with $G$ finite). For readers familiar
with their paper, let us note that $\vec{A}\odot\vec{B}$ (in their notation)
corresponds to our $\widetilde{x+y}$. They obtain the formula%
\[
h(x,y,z)=\sum_{i}n_{i}x_{i}\left\lfloor \frac{y_{i}+z_{i}}{n_{i}}\right\rfloor
\sigma_{i,i}%
\]
at the end of \cite[\S 3]{MR2755172}. This is Equation \ref{lchiav1}, again
with the summands $i\neq j$ removed by the same argument as in Equation
\ref{lviaps4}.
\end{example}

\section{Abelian $3$-cocycle formulas in exponential format}

Aside from Quinn's formula, a lot of literature prefers to explicitly spell
out the abelian $3$-cocycle in terms of exponential functions when
$M:=\mathbb{C}^{\times}$. Let us also provide this.

\begin{theorem}
[Exponential format $3$-cocycles]\label{thm_ExplicitAbelian3Cocycles}Suppose%
\begin{equation}
G=\bigoplus_{k\in J}\mathbb{Z}/n_{k}\mathbb{Z} \label{lviaps8}%
\end{equation}
for $n_{k}\geq1$ and $J$ some totally ordered index set. Write $(e_{k})_{k\in
J}$ for the generator $1$ of the $k$-th summand. Then there is a bijection
between the following three sets:

\begin{enumerate}
\item All possible choices of values

\begin{itemize}
\item $p^{(k)}\in\{0,1,\ldots,\gcd(n_{k}^{2},2n_{k})-1\}$ for every $k\in J$,

\item $q^{(k,l)}\in\{0,1,\ldots,\gcd(n_{k},n_{l})-1\}$ for all $k<l$ with
$k,l\in J$.
\end{itemize}

\item All quadratic forms $q\in\operatorname*{Quad}(G,\mathbb{C}^{\times})$,
uniquely described by the following properties%
\begin{align*}
q(e_{k})  &  =\exp\left(  \frac{2\pi i}{\gcd(n_{k}^{2},2n_{k})}p^{(k)}\right)
\text{,}\\
b(e_{k},e_{l})  &  =\exp\left(  \frac{2\pi i}{\gcd(n_{k},n_{l})}%
q^{(k,l)}\right)  \qquad\text{(for }k<l\text{),}%
\end{align*}
where $b$ is the polarization of $q$ (and further we necessarily then have
$b(e_{k},e_{l})=b(e_{l},e_{k})$ for $k>l$ and $b(e_{k},e_{k})=2q(e_{k})$ as well).

\item All abelian $3$-cocycles $(h,c)\in H_{ab}^{3}(G,\mathbb{C}^{\times})$,
uniquely pinned down by the cocycle representative%
\begin{align}
c(x,y)  &  =\prod_{k<l}\exp\left(  \frac{2\pi iq^{(k,l)}}{\gcd(n_{k},n_{l}%
)}x_{k}y_{l}\right) \label{lviaps9}\\
&  \qquad\cdot\prod_{k}\exp\left(  \frac{2\pi ip^{(k)}}{\gcd(2n_{k},n_{k}%
^{2})}x_{k}y_{k}\right)  \text{,}\nonumber
\end{align}
and%
\[
h(x,y,z)=\prod_{k}\exp\left(  \frac{2\pi ip^{(k)}}{\gcd(2n_{k},n_{k}^{2}%
)}\left(  x_{k}\left(  [y_{k}]_{n_{k}}+[z_{k}]_{n_{k}}-[y_{k}+z_{k}]_{n_{k}%
}\right)  \right)  \right)  \text{,}%
\]
where $x_{k}$ (resp. $y_{k},z_{k}$) denotes the coordinates of vectors
$x,y,z\in G$ according to Equation \ref{lviaps8}. Here $[-]_{n_{k}}$ refers to
the remainder of division by $n_{k}$, expressed as an element in
$\{0,1,\ldots,n_{k}-1\}$.
\end{enumerate}

Really, $\operatorname*{Quad}(G,\mathbb{C}^{\times})$ and $H_{ab}%
^{3}(G,\mathbb{C}^{\times})$ are abelian groups and the above bijections are
abelian group isomorphisms, given in terms of the parameters $p^{(k)}%
,q^{(k,l)}$ by elementwise addition in the quotient groups (i.e.
$\mathbb{Z}/(n_{k}^{2},2n_{k})$ for $p^{(k)}$ etc.).\newline The map
$q\mapsto(h,c)$ is linear, so it provides a group homomorphism
$\operatorname*{Quad}(G,M)\rightarrow Z_{ab}^{3}(G,M)$, which makes Diagram
\ref{l_fig1} commute.
\end{theorem}

Before we prove this, let us first establish an explicit parametrization of
the quadratic forms.

We apologize that the following repeats part of the statement of the above
theorem, but we prefer to be clear about what we prove here amidst a lot of notation.

\begin{lemma}
\label{lemma_ParametrizeQuadForms}Suppose%
\[
G=\bigoplus_{k\in J}\mathbb{Z}/n_{k}\mathbb{Z}%
\]
for $n_{k}\geq1$ and $J$ some totally ordered index set. Then all elements of
$\operatorname*{Quad}(G,\mathbb{C}^{\times})$ are in bijection to all possible choices

\begin{enumerate}
\item $p^{(k)}\in\{0,1,\ldots,\gcd(n_{k}^{2},2n_{k})-1\}$ for every $k\in J$;

\item $q^{(k,l)}\in\{0,1,\ldots,\gcd(n_{k},n_{l})-1\}$ for all $k<l$ with
$k,l\in J$.
\end{enumerate}

Once these choices are made, the corresponding quadratic form and its
polarization satisfy%
\begin{align*}
q(e_{k})  &  =\exp\left(  \frac{2\pi i}{\gcd(n_{k}^{2},2n_{k})}p^{(k)}\right)
\text{,}\\
b(e_{k},e_{l})  &  =\exp\left(  \frac{2\pi i}{\gcd(n_{k},n_{l})}%
q^{(k,l)}\right)  \qquad\text{(for }k<l\text{)}%
\end{align*}
(and then necessarily $b(e_{k},e_{l})=b(e_{l},e_{k})$ for $k>l$ and
$b(e_{k},e_{k})=2q(e_{k})$ as well).
\end{lemma}

\begin{proof}
\textit{(Step 0) }It suffices to prove this for $J$ finite, since we can write
$G$ as the colimit over all finite subsets $J_{0}\subset J$, and
correspondingly the subset of parameters $p^{(k)},q^{(k,l)}$ with $k,l\in
J_{0}$. This is compatible under inclusion of finite subsets of $J$%
.\newline\textit{(Step 1) }So assume $J$ finite. Let $q\in\operatorname*{Quad}%
(G,\mathbb{C}^{\times})$ be arbitrary. We first claim that $q(e_{k})$ must be
a $\gcd(n_{k}^{2},2n_{k})$-torsion element in $\mathbb{C}^{\times}$, i.e.%
\begin{equation}
q(e_{k})=\exp\left(  \frac{2\pi i}{\gcd(n_{k}^{2},2n_{k})}p^{(k)}\right)
\label{ccag1}%
\end{equation}
for some uniquely determined $p^{(k)}\in\{0,1,\ldots,\gcd(n_{k}^{2}%
,2n_{k})-1\}$. The proof for this goes as follows: The generator $e_{k}$ spans
a subgroup $\iota:\mathbb{Z}/n_{k}\mathbb{Z}\hookrightarrow G$, so we can pull
the quadratic form back to this subgroup.%
\begin{equation}
\operatorname*{Quad}(G,\mathbb{C}^{\times})\overset{\iota^{\ast}%
}{\longrightarrow}\operatorname*{Quad}(\mathbb{Z}/n_{k}\mathbb{Z}%
,\mathbb{C}^{\times})\cong\operatorname*{Hom}(\mathbb{Z}/(2n_{k},n_{k}%
^{2})\mathbb{Z},\mathbb{C}^{\times}) \label{lviaps5}%
\end{equation}
The last isomorphism, the determination of quadratic forms on $\mathbb{Z}%
/n_{k}\mathbb{Z}$, goes back to Whitehead. We explain the argument: One can
further pull back along $\mathbb{Z}\twoheadrightarrow\mathbb{Z}/n_{k}%
\mathbb{Z}$ and it is easy to see that any quadratic form on $\mathbb{Z}$ must
have the shape%
\begin{equation}
q^{\prime}(x):=x^{2}m \label{lviaps6}%
\end{equation}
for some $m\in\mathbb{C}^{\times}$ (e.g., use Example \ref{example_QuadScalar}%
). Here and for the rest of this sub-argument, we stick to the additive
notation. One then only needs to check for which $m$ such a $q^{\prime}$
descends to a quadratic form on the quotient $\mathbb{Z}/n_{k}\mathbb{Z}$,
which will then give Equation \ref{lviaps5}. We claim that this holds whenever
$m$ is a $\gcd(2n_{k},n_{k}^{2})$-torsion element in $\mathbb{C}^{\times}$.

Necessity: Suppose it descends. Then $0=q^{\prime}(0)=q^{\prime}(n_{k}%
)=n_{k}^{2}q^{\prime}(1)$, again by Example \ref{example_QuadScalar}. Hence,
we must have $n_{k}^{2}m=0$ in $\mathbb{C}^{\times}$. Moreover, if $q^{\prime
}$ descends to $\mathbb{Z}/n_{k}\mathbb{Z}$, so does the polarization
$b^{\prime}:G\otimes_{\mathbb{Z}}G\rightarrow\mathbb{C}^{\times}$. Hence,%
\[
b^{\prime}(x,y)=2xym\text{,}%
\]
satisfies $b^{\prime}(x+Nn_{k},y)\equiv b^{\prime}(x,y)\,\left(
\operatorname{mod}n_{k}\right)  $ for all $N$. This forces $2n_{k}m=0$. Thus,
$m$ must be both $2n_{k}$- and $n_{k}^{2}$-torsion, i.e. $\gcd(2n_{k}%
,n_{k}^{2})$-torsion in $\mathbb{C}^{\times}$. Conversely, suppose it is. Then%
\begin{align*}
q^{\prime}(x+Nn_{k})  &  =(x+Nn_{k})^{2}m\\
&  =x^{2}m+2n_{k}Nxm+N^{2}n_{k}^{2}m\equiv x^{2}m
\end{align*}
since the second and third summand vanish because of the torsion assumption.

This proves the subclaim.\ That is: the valid choices for $m\in\mathbb{C}%
^{\times}$ in Equation \ref{lviaps6} are precisely the $\gcd(2n_{k},n_{k}%
^{2})$-torsion elements. However, these are precisely such as described in
Equation \ref{ccag1}. Next, by bilinearity $nb(e_{k},e_{l})=1$ once
$\gcd(n_{k},n_{l})\mid n$ (because $e_{k}$ is $n_{k}$-torsion and $e_{l}$ is
$n_{l}$-torsion), so we may write%
\[
b(e_{k},e_{l})=\exp\left(  \frac{2\pi i}{\gcd(n_{k},n_{l})}q^{(k,l)}\right)
\]
for some uniquely determined $q^{(k,l)}\in\{0,1,\ldots,\gcd(n_{k},n_{l}%
)-1\}$.\newline\textit{(Step 2)} We have now seen that $q(e_{k})$ and
$b(e_{k},e_{l})$ can only be of the described forms. This defines a
set-theoretic map%
\begin{equation}
\operatorname*{Quad}(G,\mathbb{C}^{\times})\longrightarrow\left\{
\begin{array}
[c]{l}%
\text{parameter values }p^{(k)}\text{,}q^{(k,l)}\\
\text{(in the ranges described)}%
\end{array}
\right\}  \text{.} \label{lviaps7}%
\end{equation}
It is also clear that this map is injective. It remains to check that
conversely any choice of parameters defines a quadratic form on $G$. Since we
already have an injective map, it suffices to count elements on the right
side. For any abelian groups $A,B$ we have%
\[
\operatorname*{Quad}(A\oplus B,\mathbb{C}^{\times})\cong\operatorname*{Quad}%
(A,\mathbb{C}^{\times})\oplus\operatorname*{Quad}(B,\mathbb{C}^{\times}%
)\oplus\operatorname*{Hom}(A\otimes_{\mathbb{Z}}B,\mathbb{C}^{\times})\text{.}%
\]
To see this, note that%
\[
\operatorname*{Quad}(G,\mathbb{C}^{\times})=\operatorname*{Hom}(\Gamma
G,\mathbb{C}^{\times})
\]
for Whitehead's universal quadratic functor $\Gamma$ and the latter is a
quadratic functor, i.e.%
\[
\Gamma(A\oplus B)\cong\Gamma(A)\oplus\Gamma(B)\oplus\left(  A\otimes B\right)
\text{.}%
\]
For details, we refer to \cite[Chapter I, \S 4]{MR1096295} (or the classic
\cite{MR35997}). Hence,%
\begin{align*}
\operatorname*{Quad}\left(  \bigoplus_{k}\mathbb{Z}/n_{k}\mathbb{Z}%
,\mathbb{C}^{\times}\right)   &  \cong\bigoplus_{k}\operatorname*{Quad}%
(\mathbb{Z}/n_{k}\mathbb{Z},\mathbb{C}^{\times})\\
&  \qquad\oplus\bigoplus_{k<l}\operatorname*{Hom}(\mathbb{Z}/\gcd(n_{k}%
,n_{l})\mathbb{Z},\mathbb{C}^{\times})
\end{align*}
and this unravels to%
\[
\cong\bigoplus_{k}\mathbb{Z}/(2n_{k},n_{k}^{2})\mathbb{Z}\oplus\bigoplus
_{k<l}\mathbb{Z}/\gcd(n_{k},n_{l})\mathbb{Z}\text{.}%
\]
For the first type of summands we have again used the isomorphism of Equation
\ref{lviaps5}, and for the second type of summand note that any linear map of
a torsion group to $\mathbb{C}^{\times}$ must have its image in the torsion of
$\mathbb{C}^{\times}$ and thus certainly in $U(1)$, and $\operatorname*{Hom}%
(-,U(1))$ is just the Pontryagin dual, which (non-canonically) can be
identified with the input abelian group. Without spelling out the actual
count, it is clear that this set has the same cardinality as our set of
parameter values on the right side in Equation \ref{lviaps7}. This finishes
the proof.
\end{proof}

Now the proof of Theorem \ref{thm_ExplicitAbelian3Cocycles} can be done in a
similar fashion to the one we used to obtain Quinn's formula.

\begin{proof}
[Proof of Theorem \ref{thm_ExplicitAbelian3Cocycles}]We follow the same proof
as for Theorem \ref{thm_QuinnFormula}, so let us just describe how certain
details need to be changed.\ Firstly, we are now in the special case
$M:=\mathbb{C}^{\times}$. Using our parametrization of quadratic forms of
Lemma \ref{lemma_ParametrizeQuadForms}, we may rewrite Equation \ref{lziv3} in
the concrete shape
\begin{equation}
\sigma_{k,l}=C(e_{k},e_{l})=\left\{
\begin{array}
[c]{ll}%
\exp\left(  \frac{2\pi i}{\gcd(n_{k},n_{l})}q^{(k,l)}\right)  & \text{if
}k<l\\
\exp\left(  \frac{2\pi i}{\gcd(2n_{k},n_{k}^{2})}p^{(k)}\right)  & \text{if
}k=l\\
0 & \text{if }k>l\text{.}%
\end{array}
\right.  \label{lviaps10}%
\end{equation}
Write $[x]_{n}$ for the remainder in $\{0,1,\ldots,n-1\}$ of $x\in\mathbb{Z}$
under division by $n$. Rewrite the admissible lifting in Equation
\ref{lviaps1} in the shape%
\[
(\widetilde{x})_{j}=[x]_{n_{j}}\text{.}%
\]
This is an admissible lift in the sense of Definition \ref{def_AL} (and only
optically different from the choice in the proof of Quinn's formula). We
compute%
\[
L(x,y)_{j}=[x_{j}+y_{j}]_{n_{j}}-[x_{j}]_{n_{j}}-[y_{j}]_{n_{j}}%
\]
and obtain%
\begin{align*}
h(x,y,z)  &  =\prod_{k\leq l}\sigma_{k,l}^{x_{k}\left(  [y_{l}]_{n_{l}}%
+[z_{l}]_{n_{l}}-[y_{l}+z_{l}]_{n_{l}}\right)  }\\
c(x,y)  &  =\prod_{k\leq l}\sigma_{k,l}^{x_{k}y_{l}}%
\end{align*}
because of the multiplicative notation in $\mathbb{C}^{\times}$. We readily
read off Equation \ref{lviaps9} for $c(x,y)$, just by plugging in the values
of $\sigma_{k,l}$ as provided by Equation \ref{lviaps10}. Similarly,%
\begin{align*}
h(x,y,z)  &  =\left(  \prod_{k<l}\exp\left(  \frac{2\pi iq^{(k,l)}}{\gcd
(n_{k},n_{l})}\left(  x_{k}\left(  [y_{l}]_{n_{l}}+[z_{l}]_{n_{l}}%
-[y_{l}+z_{l}]_{n_{l}}\right)  \right)  \right)  \right) \\
&  \qquad\cdot\left(  \prod_{k}\exp\left(  \frac{2\pi ip^{(k)}}{\gcd
(2n_{k},n_{k}^{2})}\left(  x_{k}\left(  [y_{k}]_{n_{k}}+[z_{k}]_{n_{k}}%
-[y_{k}+z_{k}]_{n_{k}}\right)  \right)  \right)  \right)  \text{.}%
\end{align*}
As in Equation \ref{lviaps4}, for $k<l$ the expression $[y_{l}]_{n_{l}}%
+[z_{l}]_{n_{l}}-[y_{l}+z_{l}]_{n_{l}}$ is a multiple of $n_{l}$, so the
entire input to the exponential function lies in $2\pi i\mathbb{Z}$. Thus,%
\[
h(x,y,z)=\prod_{k}\exp\left(  \frac{2\pi ip^{(k)}}{\gcd(2n_{k},n_{k}^{2}%
)}\left(  x_{k}\left(  [y_{k}]_{n_{k}}+[z_{k}]_{n_{k}}-[y_{k}+z_{k}]_{n_{k}%
}\right)  \right)  \right)  \text{,}%
\]
which is what we had claimed.
\end{proof}

\section{Constructing associators and
braidings\label{sect_AssociatorsAndBraidings}}

\subsection{Recollections}

As we had explained in the introduction, this note applies to both (a) pointed
braided fusion categories over a field $k$, as well as (b) braided categorical groups.

The reason for this is that both types of braided monoidal categories can be
classified in terms of (a slight generalization of) pre-metric groups as in
\cite{MR1250465}, \cite{MR2609644}.

\begin{definition}
[{\cite[\S 3]{MR1250465}}]\label{def_QuadTriple}A \emph{quadratic triple} is a
triple $(G,M,q)$, where $G,M$ are abelian groups and $q\in\operatorname*{Quad}%
(G,M)$ a quadratic form. A morphism $(G,M,q)\rightarrow(G^{\prime},M^{\prime
},q^{\prime})$ is a commutative diagram%
\begin{equation}%
\xymatrix{
G \ar[r]^{f} \ar[d]_{q} & G^{\prime} \ar[d]^{q^{\prime}} \\
M \ar[r]_{g} & M^{\prime},
}
\label{lviapsA2}%
\end{equation}
where $f,g$ are group homomorphisms. Write $\mathcal{Q}uad$ for the category
of quadratic triples. A \emph{symmetric triple} is a quadratic triple such
that the polarization of the quadratic form $q$ vanishes. Write $\mathcal{Q}%
uad_{sym}$ for the full subcategory of symmetric triples.
\end{definition}

Write $\mathcal{BCG}$ for the $1$-category whose objects are braided
categorical groups and whose morphisms are the equivalence classes of braided
monoidal functors.\footnote{In \cite{MR1250465} the category $\mathcal{BCG}$
is defined without identifying morphisms which only differ by equivalence.
This leads to the statement of the classification \cite[Theorem 3.3]%
{MR1250465} to sound more convoluted.}

\begin{theorem}
[{Joyal--Street \cite[\S 3]{MR1250465}}]%
\label{thm_JoyalStreetEquivBCGAndQuadTriples}There is an equivalence of
$1$-categories%
\begin{align}
T:\mathcal{BCG}  &  \longrightarrow\mathcal{Q}uad\label{lviapsA1}\\
(\mathsf{C},\otimes)  &  \longmapsto(\pi_{0}(\mathsf{C},\otimes),\pi
_{1}(\mathsf{C},\otimes),q)\text{,}\nonumber
\end{align}
where $q$ is defined as follows: For any object $X\in\mathsf{C}$ the
self-braiding $s_{X,X}:X\otimes X\overset{\sim}{\rightarrow}X\otimes X$
induces an automorphism of the tensor unit, namely%
\[
s_{X,X}\otimes X^{-1}\otimes X^{-1}:1_{\mathsf{C}}\overset{\sim}%
{\longrightarrow}1_{\mathsf{C}}\text{,}%
\]
and $q([X]):=(s_{X,X}\otimes X^{-1}\otimes X^{-1})\in\operatorname*{Aut}%
(1_{\mathsf{C}})\cong M$ extends to a well-defined quadratic form on $\pi
_{0}(\mathsf{C},\otimes)$.
\end{theorem}

The idea is as follows: Firstly, one picks a skeleton of the category, using
that every braided monoidal category is braided monoidal equivalent to any of
its skeleta (with a suitably defined braided monoidal structure on the
skeleton). On this skeleton, the associator%
\[
a_{X,Y,Z}:(X\otimes Y)\otimes Z\overset{\sim}{\longrightarrow}X\otimes
(Y\otimes Z)
\]
and braiding%
\[
s_{X,Y}:X\otimes Y\overset{\sim}{\longrightarrow}Y\otimes X
\]
are automorphisms of objects, because in a skeleton any two mutually
isomorphic objects must be the same objec. Thus, $a_{X,Y,Z}\in M$ and
$s_{X,Y}\in M$ for any objects $X,Y,Z$. The pentagon and hexagon axioms of a
braided monoidal structure then become equivalent to the axioms of an abelian
$3$-cocycle $(h,c)$ (namely Equation \ref{lefmu1a} for the pentagon, and
Equations \ref{lx_a_1}, \ref{lx_a_2} for the hexagons).

This means that up to braided monoidal equivalence it suffices to work with
the following explicit skeletal model in the situation $G:=\pi_{0}%
(\mathsf{C},\otimes)$, $M:=\pi_{1}(\mathsf{C},\otimes)=\operatorname*{Aut}%
_{\mathsf{C}}(1_{\mathsf{C}})$ and how $h,c$ enter is explained below:

\begin{definition}
[Skeletal model]\label{def_SkeletalModel}Suppose we are given abelian groups
$G,M$ and an abelian $3$-cocycle $(h,c)\in Z_{ab}^{3}(G,M)$. Let
$\mathcal{T}:=\mathcal{T}(G,M,(h,c))$ denote the following braided categorical group:

\begin{enumerate}
\item The objects are the elements in $G$.

\item The automorphisms of an object $X\in G$ are $\operatorname*{Aut}(X):=M$,
and their composition is the addition of $M$.

\item There are no morphisms except for automorphisms (so no other composition
of morphisms needs to be defined).

\item The monoidal structure is%
\[
(X\overset{f}{\longrightarrow}X)\otimes_{\mathcal{T}}(X^{\prime}%
\overset{f^{\prime}}{\longrightarrow}X^{\prime}):=(X+X^{\prime}\overset
{f+f^{\prime}}{\longrightarrow}X+X^{\prime})\text{,}%
\]
where addition is just addition in $G$ (on objects) resp. in $M$ (for
$f,f^{\prime}$).

\item The associator%
\begin{equation}
a_{X,Y,Z}:(X\otimes Y)\otimes Z\overset{\sim}{\longrightarrow}X\otimes
(Y\otimes Z) \label{lviaps3A}%
\end{equation}
is the automorphism defined by $h(X,Y,Z)\in M$. The associativity of $G$
settles that the objects on either side are the same.

\item The braiding%
\[
s_{X,Y}:X\otimes Y\longrightarrow Y\otimes X
\]
is the automorphism given by $c(X,Y)\in M$.
\end{enumerate}
\end{definition}

One then checks that changing the abelian $3$-cocycle by an abelian
$3$-coboundary amounts to the structure rendering the identity functor
$\operatorname*{id}:\mathcal{T}\rightarrow\mathcal{T}$ a braided monoidal
self-equivalence, i.e. only the cohomology class of $[(h,c)]\in H_{ab}%
^{3}(G,M)$ is well-defined when regarding $\mathcal{T}$ as an object in
$\mathcal{BCG}$. Finally, the latter identifies uniquely with a quadratic form
$q\in\operatorname*{Quad}(G,M)$ via the\ Eilenberg--Mac Lane isomorphism,
Equation \ref{lbjca1} (or \S \ref{sect_ProofOfEilenbergMacLaneThm}). For
details regarding the proof of Theorem
\ref{thm_JoyalStreetEquivBCGAndQuadTriples} we refer to \cite[\S 3]{MR1250465}
(or the slightly different treatment in \cite{joyalstreetpreprint}).

Let $\mathcal{BCG}_{sym}\subset\mathcal{BCG}$ denote the full subcategory of
those braided categorical groups which are symmetric monoidal (the braided
monoidal equivalences then automatically become symmetric monoidal
equivalences). These are also known as \emph{Picard groupoids}.

\begin{theorem}
[S\'{\i}nh \cite{sinh}]\label{thm_Sinh}The equivalence of Equation
\ref{lviapsA1} restricts to an equivalence of $1$-categories
\[
T:\mathcal{BCG}_{sym}\longrightarrow\mathcal{Q}uad_{sym}\text{.}%
\]

\end{theorem}

\begin{proof}
This was originally proven in S\'{\i}nh's thesis \cite{sinh}, and this
essentially started the entire subject from scratch. In particular, the
formulation as a special case of the Joyal--Street equivalence is
anachronistic. S\'{\i}nh defined her own functor and the right side was
different (but equivalent to $\mathcal{Q}uad_{sym}$). We explain how the above
arises as a special case within the Joyal--Street classification: If the
braiding is symmetric, we have%
\[
s_{Y,X}\circ s_{X,Y}=\operatorname*{id}\nolimits_{X\otimes Y}\text{,}%
\]
and in terms of the abelian $3$-cocycle $(h,c)$ built from the braiding and
associator this amounts to $c(y,x)+c(x,y)=0$. Among all abelian $3$-cocycles,
this additional constraint isolates the symmetric $3$-cocycles, $H_{sym}%
^{3}(G,M)\subseteq H_{ab}^{3}(G,M)$, but there is a commutative diagram, going
back to Whitehead, whose top row is exact (see \cite[Lemma 4.10]{bcg})%
\[%
\xymatrix{
0 \ar[r] & \operatorname*{Hom}(G/2G,M) \ar@{^{(}->}[r] & \operatorname
*{Quad}(G,M) \ar[r]^-{q \mapsto b} & \operatorname*{Hom}(G\otimes G,M) \\
0 \ar[r] & H^3_{sym}(G,M) \ar[u]^{\cong} \ar@{^{(}->}[r] & H^3_{ab}
(G,M) \ar[u]^{\cong}, \\
}%
\]
where \textquotedblleft$q\mapsto b$\textquotedblright\ refers to the map
sending a quadratic form to its polarization. In particular, $H_{sym}%
^{3}(G,M)$ identifies precisely with those quadratic forms whose polarization
vanishes, which was the defining property for $\mathcal{Q}uad_{sym}$ on the
right side and proves the claim. The embedding $\operatorname*{Hom}%
(G/2G,M)\hookrightarrow\operatorname*{Quad}(G,M)$ is based on the observation
that any linear map $G/2G\rightarrow M$ satisfies the axioms of a quadratic
form (see \cite[Lemma 4.10]{bcg} for details).
\end{proof}

Now return to the concept of a quadratic triple as in Definition
\ref{def_QuadTriple}. Let $k$ be a field. In the special case where $G$ is
finite, $M:=k^{\times}$ and $g$ in Diagram \ref{lviapsA2} is constrained to be
the identity map, the definition transforms into the concept of a pre-metric
group. We will mildly generalize this and drop the finiteness assumption.

\begin{definition}
For us, a \emph{big fusion category} is a semisimple rigid $k$-linear monoidal
category $(\mathsf{C},\otimes)$ with finite-dimensional $\operatorname*{Hom}%
$-spaces such that the monoidal unit $1_{\mathsf{C}}$ is simple. It is
\emph{pointed} if all simple objects are invertible.
\end{definition}

An ordinary (non-big) fusion category is a big fusion category such that there
are only finitely many isomorphism classes in $\mathsf{C}$; this is the
standard definition. The definition of a semisimple category includes that
every object decomposes as a \textit{finite} direct sum of simple objects.
Nonetheless, this would have to hold even if we only required that every
object is a direct sum of simples: If $X\simeq\bigoplus_{i\in I}S_{i}$ with
each $S_{i}$ simple (or merely non-zero), the assumption $\dim_{k}%
\operatorname*{End}(X)<\infty$ already forces the index set $I$ to be finite.

If $X$ is a simple object in a pointed big fusion category, it is invertible
and thus $X^{-1}$ is also simple; and if $X,Y$ are both simple, then $X\otimes
Y$ can only have a single simple direct summand. Thus, $X\otimes Y$ must be
simple, too.

\begin{definition}
If $(\mathsf{C},\otimes)$ is a big pointed braided fusion category, let
$\mathsf{C}_{\operatorname*{simp}}$ be the full subcategory of simple objects,
and keep only the isomorphisms as morphisms. Thus, $\mathsf{C}%
_{\operatorname*{simp}}$ is a groupoid. Moreover, $\mathsf{C}%
_{\operatorname*{simp}}$ is braided monoidal by restricting the braided
monoidal structure to the subcategory.
\end{definition}

\begin{definition}
[{\cite[Definition 2.38]{MR2609644}, \cite[\S 8.4]{MR3242743}}]Fix a field
$k$. A \emph{big pre-metric group} is a pair $(G,q)$, where $G$ is an abelian
group and $q\in\operatorname*{Quad}(G,k^{\times})$ a quadratic form. A
morphism $(G,q)\rightarrow(G^{\prime},q^{\prime})$ is a commutative diagram%
\[%
\xymatrix{
G \ar[rr]^{f} \ar[dr]_{q} & & G^{\prime} \ar[dl]^{q^{\prime}} \\
& k^{\times},
}%
\]
where $f$ is a group homomorphism. Write $\mathcal{PM}_{k}$ for the category
of big pre-metric groups. A \emph{symmetric big pre-metric group} is a pair
$(G,q)$ such that the polarization of the quadratic form vanishes.\ Write
$\mathcal{PM}_{k,sym}$ for the corresponding full subcategory. The original
definition without the word \textquotedblleft big\textquotedblright\ refers to
the same, except that we demand $G$ to be a finite abelian group.
\end{definition}

Write $\mathcal{PB}_{k}$ for the $1$-category whose objects are pointed
braided $k$-linear big fusion categories and whose morphisms are the
equivalence classes of $k$-linear braided monoidal functors. Write
$\mathcal{PS}_{k}\subset\mathcal{PB}_{k}$ for the full subcategory of those
big fusion categories which are symmetric monoidal.

\begin{theorem}
[{\cite[Proposition 2.41]{MR2609644}}]\label{thm_DGNO_PreMetricGroupClassif}%
Let $k$ be an algebraically closed field. There is an equivalence of
$1$-categories%
\begin{align*}
E:\mathcal{PB}_{k}  &  \longrightarrow\mathcal{PM}_{k}\\
(\mathsf{C},\otimes)  &  \longmapsto(\pi_{0}(\mathsf{C},\otimes),q)\text{,}%
\end{align*}
where $q$ is defined as follows: For any simple object $X\in\mathsf{C}$ the
self-braiding $s_{X,X}:X\otimes X\overset{\sim}{\rightarrow}X\otimes X$
induces an automorphism of the tensor unit, namely%
\[
s_{X,X}\otimes X^{-1}\otimes X^{-1}:1_{\mathsf{C}}\overset{\sim}%
{\longrightarrow}1_{\mathsf{C}}\text{,}%
\]
and $q([X]):=(s_{X,X}\otimes X^{-1}\otimes X^{-1})\in\operatorname*{Aut}%
(1_{\mathsf{C}})\cong k^{\times}$ extends to a well-defined quadratic form on
$\pi_{0}(\mathsf{C},\otimes)$.
\end{theorem}

\begin{proof}
This result essentially reduces to the classification of Joyal--Street of
braided categorical groups. The main differences are as follows: (a) Since
$1_{\mathsf{C}}$ is simple by assumption, $\operatorname*{End}(1_{\mathsf{C}%
})$ must be a division algebra over $k$ by Schur's Lemma. As $k$ is
algebraically closed, we must have $\operatorname*{End}(1_{\mathsf{C}})=k$ and
thus $\operatorname*{Aut}(1_{\mathsf{C}})=k^{\times}$. The category
$\mathsf{C}_{\operatorname*{simp}}$ thus is a braided categorical group with
$\pi_{1}(\mathsf{C}_{\operatorname*{simp}})=k^{\times}$. (b) Since the braided
monoidal functors between big fusion categories are assumed $k$-linear, they
correspond to morphisms of quadratic triples as in Figure \ref{lviapsA2}, but
must be the identity on $k^{\times}$.
\end{proof}

See also \cite[Theorem 8.4.12]{MR3242743} for a direct proof which differs in
certain parts. In analogy to Theorem \ref{thm_Sinh} one obtains the following.

\begin{theorem}
[{\cite[Example 2.45]{MR2609644}}]The equivalence of Equation \ref{lviapsA1}
restricts to an equivalence of $1$-categories
\[
T:\mathcal{PS}_{k}\longrightarrow\mathcal{PM}_{k,sym}\text{.}%
\]

\end{theorem}

\subsection{Structure results}

Now, let us draw some conclusions from these equivalences. The following is
certainly known in the setting of fusion categories.

\begin{theorem}
\label{thm_characterize_simultaneouslySkeletalAndStrictlyAssociative}Let
$(\mathsf{C},\otimes)$ be

\begin{description}
\item[(a)] a $k$-linear pointed braided fusion category with $k$ algebraically
closed, or

\item[(b)] a braided categorical group.
\end{description}

Write $\pi_{i}:=\pi_{i}(\mathsf{C}_{\operatorname*{simp}},\otimes)$ in
situation (a), and $\pi_{i}:=\pi_{i}(\mathsf{C},\otimes)$ in situation (b).
Then the following statements are equivalent:

\begin{enumerate}
\item $(\mathsf{C},\otimes)$ is braided monoidal equivalent to a
simultaneously skeletal and strictly associative braided monoidal category.

\item The abelian $3$-cocycle of $(\mathsf{C},\otimes)$ admits a cocycle
representative $(h,c)$ such that%
\[
c:\pi_{0}\times\pi_{0}\longrightarrow\pi_{1}%
\]
is bilinear and $h$ vanishes.

\item The abelian $3$-cocycle of $(\mathsf{C},\otimes)$ admits a cocycle
representative $(h,c)$ such that%
\[
h:\pi_{0}\times\pi_{0}\times\pi_{0}\longrightarrow\pi_{1}\qquad\text{and}%
\qquad c:\pi_{0}\times\pi_{0}\longrightarrow\pi_{1}%
\]
are trilinear resp. bilinear in $\pi_{0}$.

\item There exists a bilinear form $S$ on $\pi_{0}$ such that%
\[
q(x)=S(x,x)
\]
holds for the quadratic form attached to $(h,c)$ under the Eilenberg--Mac Lane
isomorphism of Equation \ref{ljbc6a}.
\end{enumerate}
\end{theorem}

The equivalence (1)$\Leftrightarrow$(4) in the case of fusion categories is
found in \cite[Exercise 8.4.11]{MR3242743}.

\begin{proof}
We discuss the braided categorical group case. Suppose $(G,M,q)$ is the
quadratic triple attached to $(\mathsf{C},\otimes)$ by Theorem
\ref{thm_JoyalStreetEquivBCGAndQuadTriples} $(1\Rightarrow2)$ Suppose
$(\mathsf{C},\otimes)$ is braided monoidal equivalent to a skeletal and
strictly associative braided monoidal category. Being skeletal, it is of the
form of a skeletal model $\mathcal{T}(G,M,(h,c))$, Definition
\ref{def_SkeletalModel}, where $(h,c)$ is an abelian $3$-cocycle. Since the
model is strictly associative, the maps%
\[
a_{X,Y,Z}:(X\otimes Y)\otimes Z\overset{\sim}{\longrightarrow}X\otimes
(Y\otimes Z)
\]
of Equation \ref{lviaps3A} are identity maps in this model, i.e. $h(x,y,z)=0$
holds for all $x,y,z$. Now Equations \ref{lx_a_1}, \ref{lx_a_2} imply that $c$
is $\mathbb{Z}$-bilinear. $(2\Rightarrow3)$ Trivial. $(3\Rightarrow4)$ Define
$S(x,y):=c(x,y)$. This is bilinear by assumption. Moreover,%
\[
S(x,x)=c(x,x)=q(x)
\]
since the Eilenberg--Mac Lane isomorphism maps $(h,c)$ to $x\mapsto c(x,x)$.
$(4\Rightarrow1)$ We can apply Lemma \ref{lemma_AP_QuadFormFromBilinearForm}
(2), so $(G,\operatorname*{id}_{G},S)$ is an optimal admissible presentation
with $F_{1}=0$. According to Theorem \ref{thm_GeneralFormulaWithAP} it follows
that $(h,c)$ with%
\[
h(x,y,z):=-S(\widetilde{x},L(y,z))\text{,}\qquad c(x,y):=S(\widetilde
{x},\widetilde{y})
\]
is an abelian $3$-cocycle representative for the quadratic form $q$ for any
admissible lifting $\widetilde{(-)}$. Since $F_{0}=G$, we can just take the
identity as a lifting. In particular, $L(x,y)=\widetilde{(x+y)}-\widetilde
{x}-\widetilde{y}=0$, so $h(x,y,z)=0$ for all $x,y,z\in G$. It follows that
the skeletal model $\mathcal{T}(G,M,(0,c))$ is braided monoidal equivalent to
$(\mathsf{C},\otimes)$, and this is skeletal, and thanks to $h=0$ the
associators are trivial. In the big fusion category case, instead work with
the triple $(G,k^{\times},q)$ of Theorem \ref{thm_DGNO_PreMetricGroupClassif},
and replace the skeletal model $\mathcal{T}(G,M,(h,c))$ by its natural
$k$-linear analogue.
\end{proof}

In the symmetric monoidal situation the equivalent properties of Theorem
\ref{thm_characterize_simultaneouslySkeletalAndStrictlyAssociative} are always satisfied.

\begin{theorem}
[Johnson--Osorno \cite{MR2981952}]Let $(\mathsf{C},\otimes)$ be

\begin{itemize}
\item a $k$-linear pointed symmetric fusion category with $k$ algebraically
closed, or

\item a Picard groupoid.
\end{itemize}

Then $(\mathsf{C},\otimes)$ is symmetric monoidal equivalent to a
simultaneously skeletal and strictly associative symmetric monoidal category.
More precisely, all characterizations in Theorem
\ref{thm_characterize_simultaneouslySkeletalAndStrictlyAssociative} are always
met in this setting.
\end{theorem}

This was originally observed in the context of Picard groupoids \cite[Theorem
2.2]{MR2981952} with a different method.

\begin{proof}
We apply the equivalence of categories in the symmetric setting, Theorem
\ref{thm_Sinh}, so $(\mathsf{C},\otimes)$ corresponds to a symmetric triple
$(G,M,q)$ in $\mathcal{Q}uad_{sym}$, or $(G,k^{\times},q)$ in the fusion
category setting. Being in $\mathcal{Q}uad_{sym}$, we know that the
polarization of $q$ is trivial, i.e.%
\[
0=b(x,y)=q(x+y)-q(x)-q(y)\text{,}%
\]
showing that $q$ is a linear map (the complete deduction of this is as we had
done it around Equation \ref{lwdaa1}). As the above correspondence of Theorem
\ref{thm_Sinh} is just a special case of the Joyal--Street equivalence of
Theorem \ref{thm_JoyalStreetEquivBCGAndQuadTriples}, we can also use Theorem
\ref{thm_characterize_simultaneouslySkeletalAndStrictlyAssociative} in this
setting. The conclusion (4)$\Rightarrow$(1) shows that our claim is proven if
we can show that $q$ comes from a bilinear form. Define%
\begin{equation}
S(x,y)=\sum_{i}x_{i}y_{i}q(\gamma_{i}) \label{lcuac1}%
\end{equation}
with respect to the coordinates $x_{i},y_{i}\in\mathbb{F}_{2}$ in any chosen
basis $(\gamma_{i})_{i\in I}$ of $G/2G$ as an $\mathbb{F}_{2}$-vector space.
Since $q(-\gamma_{i})=q(\gamma_{i})$ holds for any quadratic form, the
linearity of $q$ yields $2q(\gamma_{i})=0$, so the scalar multiplication with
elements from $\mathbb{F}_{2}$ in Equation \ref{lcuac1} is well-defined and
$S$ is indeed bilinear. Then $S(x,x)=\sum x_{i}^{2}q(\gamma_{i})\equiv\sum
x_{i}q(\gamma_{i})=q(x)$, proving the claim. The remaining characterizations
follow since by Theorem
\ref{thm_characterize_simultaneouslySkeletalAndStrictlyAssociative} they are
all equivalent.
\end{proof}

As a further refinement of the characterization in Theorem
\ref{thm_characterize_simultaneouslySkeletalAndStrictlyAssociative} relating
the possibility to trivialize associators to the linearity of the braiding, we
can show the following.

\begin{theorem}
\label{thm_StdForm}Suppose $(\mathsf{C},\otimes)$ is a braided categorical
group such that $\pi_{1}(\mathsf{C},\otimes)$ is a divisible group.\ Then
$(\mathsf{C},\otimes)$ is braided monoidal equivalent to a skeletal model as
in Definition \ref{def_SkeletalModel} such that%
\[
h(x,y,z)=c(x,y)+c(x,z)-c(x,y+z)
\]
holds for the abelian $3$-cocycle. That is: For any objects $X,Y,Z$ we have%
\begin{equation}
a_{X,Y,Z}=s_{X,Y}+s_{X,Z}-s_{X,Y\otimes Z}\qquad\text{and}\qquad
a_{Z,X,Y}=s_{X\otimes Y,Z}-s_{X,Z}-s_{Y,Z}\text{.} \label{lcce4}%
\end{equation}
We call this a \emph{normal form skeletal model}.
\end{theorem}

We do not know whether this is also true for general $\pi_{1}(\mathsf{C}%
,\otimes)$.

\begin{proof}
By Theorem \ref{thm_AbstractExistence} since $M:=\pi_{1}(\mathsf{C},\otimes)$
is divisible, an optimal admissible presentation $(F_{0},\pi,C)$ exists. Let
$\widetilde{(-)}$ be any admissible lift for this presentation. Then by
Theorem \ref{thm_GeneralFormulaWithAP} we find an abelian $3$-cocycle
representative $(h,c)$ such that%
\begin{align*}
h(x,y,z)  &  =-C(\widetilde{x},L(y,z))=C(\widetilde{x},\widetilde
{y})+C(\widetilde{x},\widetilde{z})-C(\widetilde{x},\widetilde{(y+z)})\\
&  =c(x,y)+c(x,z)-c(x,y+z)\text{,}%
\end{align*}
which is precisely our claim. Combine it with Theorem
\ref{thm_JoyalStreetEquivBCGAndQuadTriples} and Definition
\ref{def_SkeletalModel} to relate it to the concrete associator and braiding.
This yields $a_{X,Y,Z}=s_{X,Y}+s_{X,Z}-s_{X,Y\otimes Z}$. For the second
identity,%
\begin{align*}
c(x+y,z)-c(x,z)-c(y,z)  &  =C(\widetilde{(x+y)},\widetilde{z})-C(\widetilde
{x},\widetilde{z})-C(\widetilde{y},\widetilde{z})=C(L(x,y),\widetilde{z})\\
&  =b(\pi L(x,y),\pi\widetilde{z})-C(\widetilde{z},L(x,y))=h(z,x,y)
\end{align*}
since $\pi L(y,z)=0$ holds by construction.
\end{proof}

Analogously, we obtain:

\begin{theorem}
\label{thm_AnnFull}Suppose $k$ is an algebraically closed field. Let
$(\mathsf{C},\otimes)$ be a $k$-linear pointed braided big fusion category.
Then $(\mathsf{C},\otimes)$ is braided monoidal equivalent to a skeletal big
fusion category such that%
\[
a_{X,Y,Z}=\frac{s_{X,Y}\cdot s_{X,Z}}{s_{X,Y\otimes Z}}\qquad\text{and}\qquad
a_{Z,X,Y}=\frac{s_{X\otimes Y,Z}}{s_{X,Z}\cdot s_{Y,Z}}%
\]
hold for all simple objects $X,Y,Z$.
\end{theorem}

\begin{proof}
The only difference to the proof for Theorem \ref{thm_StdForm} is that we need
to show that $\pi_{1}(\mathsf{C},\otimes)=k^{\times}$ is divisible, but this
just amounts to the existence of $n$-th roots $\sqrt[n]{x}$ for any $n\geq1$
and $x\in k^{\times}$, which holds since $k$ is algebraically closed.
\end{proof}

\section{Classification of monoidal structures on graded vector
spaces\label{sect_ClassifSmallExamples}}

Let $G$ be a finite abelian group and $k$ a field. Write $\mathsf{Vect}%
_{k}^{G}$ for the category of $G$-graded finite-dimensional $k$-vector spaces.

The previous results can be used to achieve a classification of all

\begin{itemize}
\item monoidal structures on $\mathsf{Vect}_{k}^{G}$ up to monoidal equivalence,

\item braided monoidal structures on $\mathsf{Vect}_{k}^{G}$ up to braided
monoidal equivalence,

\item symmetric monoidal structures on $\mathsf{Vect}_{k}^{G}$ up to symmetric
monoidal equivalence.
\end{itemize}

Bulacu, Caenepeel and Torrecillas have settled this classification for
$G=\mathbb{Z}/2\mathbb{Z}\oplus\mathbb{Z}/2\mathbb{Z}$ in \cite{MR2837470}.
They found $8$ braided monoidal structures if $k$ does not contain a primitive
$4$-th root of unity and $24$ more otherwise. We recover this in Theorem
\ref{thm_ExplicitAbelian3Cocycles}, as the set of choices for the parameters is

\begin{enumerate}
\item $p^{(0)},p^{(1)}\in\{0,1,2,3\}$,

\item $q^{(0,1)}\in\{0,1\}$,
\end{enumerate}

giving a total of $4\cdot4\cdot2=32$ choices. The theorem also gives the
respective braiding and associator and it is easy to distinguish the two
families depending on whether they output a primitive $4$-th root of unity or
not. The paper \cite{MR3228486} by Huang, Liu and Ye generalized this to the
situation $G=\mathbb{Z}/n_{1}\mathbb{Z}\oplus\mathbb{Z}/n_{2}\mathbb{Z}$ for
arbitrary $n_{1},n_{2}$, see \cite[Theorem 4.2]{MR3228486} (and note that the
set $A_{a,b,d}$ loc. cit. can be empty, as is also stressed there).

\section{\label{sect_ProofOfEilenbergMacLaneThm}Proof of the Eilenberg-Mac
Lane isomorphism}

As we are so close now, let us quickly sketch how one can reprove the original
Eilenberg-Mac Lane isomorphism along these lines. Our method does not add
anything new to checking injectivity though.

\begin{theorem}
[Eilenberg--Mac Lane]Let $G,M$ be abelian groups. The trace%
\begin{align}
H_{ab}^{3}(G,M)  &  \longrightarrow\operatorname*{Quad}(G,M)\label{ljbc6a}\\
(h,c)  &  \longmapsto(q(x):=c(x,x))\nonumber
\end{align}
is an isomorphism of abelian groups.
\end{theorem}

\begin{proof}
\textit{(Well-defined)} First of all, given an abelian $3$-cocycle $(h,c)$, we
need to check that $q(x):=c(x,x)$ is indeed a quadratic form. A\ complete
argument with all details is spelled out in \cite[Lemma 3.9]{bcg}. Next, we
have to check that cohomologous cocycles have the same trace. However, by
Equation \ref{lefmu2a} the diagonal terms $c(x,x)$ of an abelian
$3$-coboundary vanish. Thus, the map is a well-defined group
homomorphism.\textit{ (Surjective)}\ If $G$ is finitely generated abelian, the
generalized form of Quinn's formula, Theorem \ref{thm_QuinnFormula}, gives an
explicit preimage for any quadratic form $q$. An arbitrary $G$ is the
(filtering) colimit of its finitely generated subgroups $G^{\prime}$,
partially ordered by inclusion, so%
\[
\underset{G^{\prime}\subseteq G}{\underleftarrow{\lim}}H_{ab}^{3}(G^{\prime
},M)\longrightarrow\underset{G^{\prime}\subseteq G}{\underleftarrow{\lim}%
}\operatorname*{Quad}(G^{\prime},M)
\]
is also surjective in the limit.\textit{ (Injective)} We do not know a way to
improve on the argument of Joyal and Street in \cite[Theorem 12]%
{joyalstreetpreprint}. The idea is to check injectivity for $G:=\mathbb{Z}$
and $G:=\mathbb{Z}/n_{j}\mathbb{Z}$ individually, and then that both sides of
the map are quadratic functors in $G$, showing injectivity for arbitrary
finitely generated abelian groups $G$, and then use a colimit argument as
before. Surely one can run this also explicitly on the level of cocycle
computations, but since it is just about proving triviality, it is not clear
what knowing an explicit coboundary formula would then be useful for,
afterwards. So, we think the approach of Joyal and Street is just right.
\end{proof}

\section{Afterword}

The following question appears very natural to me, but I\ have not been able
to determine the answer (perhaps embarrassingly).

\begin{problem}
Suppose $G$ is a torsion-free abelian group and $M$ an arbitrary abelian
group. Is any quadratic form $q:G\rightarrow M$ of the shape $x\mapsto S(x,x)$
for $S$ some bilinear form?
\end{problem}

\bibliographystyle{amsalpha}
\bibliography{ollinewbib}

\end{document}